\documentclass[]{interact}

\usepackage{epstopdf}
\usepackage[caption=false]{subfig}

\usepackage{amssymb}
\usepackage [latin1]{inputenc}
\usepackage{amsmath,amssymb}
\usepackage{marvosym,mathtools}
\usepackage[numbers,sort&compress]{natbib}
\usepackage[colorlinks,linkcolor=blue,urlcolor=blue,citecolor=blue]{hyperref}
\usepackage{graphicx,subfig}
\usepackage{float}
\usepackage{graphicx}

\usepackage[numbers,sort&compress]{natbib}
\bibpunct[, ]{[}{]}{,}{n}{,}{,}

\theoremstyle{plain}
\newtheorem{theorem}{Theorem}[section]
\newtheorem{lemma}[theorem]{Lemma}
\newtheorem{corollary}[theorem]{Corollary}

\theoremstyle{definition}

\newtheorem{example}[theorem]{Example}

\theoremstyle{remark}
\newtheorem{remark}{Remark}

\begin{document}


\title{Inertial primal-dual dynamics with Hessian-driven damping and Tikhonov regularization for convex-concave bilinear saddle point problems}

\author{
\name{Xiangkai Sun, \thanks{CONTACT Xiangkai Sun. Email: sunxk@ctbu.edu.cn} Liang He and Xian-Jun Long }
\affil{Chongqing Key Laboratory of  Statistical Intelligent Computing and Monitoring, School of Mathematics and Statistics,
 Chongqing Technology and Business University,
Chongqing 400067, China }
}

\maketitle

\begin{abstract}
This paper deals with a second-order  primal-dual dynamical system with Hessian-driven damping and Tikhonov regularization terms in connection with a convex-concave bilinear saddle point problem. We  first obtain a fast convergence rate of  the primal-dual gap along the trajectory generated by the dynamical system, and provide some integral estimates. Then,   based on the setting of the parameters involved,  we demonstrate that both the convergence rate of the primal-dual gap and the strong convergence of the trajectory can be achieved simultaneously.  Furthermore, we  evaluate the performance of the proposed system using two numerical examples.
\end{abstract}


\amscodename\  34D05; 37N40; 46N10; 90C25.

\section{Introduction}

In recent years, there has been substantial interest in studying the close link between second-order continuous-time  dynamical systems and optimization algorithms. On the one hand, by developing the asymptotic properties of second-order dynamical systems, one can efficiently solve optimization problems. On the other hand, it  provides thoughtful insights to understand the existing optimization algorithms from a continuous-time perspective. Furthermore, an appropriate time discretization of second-order dynamical systems  can potentially lead to inertial optimization algorithms that inherit the convergence properties of  the original dynamical systems. We refer the readers to  \cite{Polyak1, su16, Attouch2018F,AttouchChbani,Attouch2022D,L2023O,Chbani2024O}
for more details.

For most of the studies in this area, we observe that  second-order dynamical systems  only  controlled by viscous damping may exhibit undesired oscillations when it is used to solving optimization problems. However,  the oscillations of trajectories can be mitigated  by virtue of  a geometric damping that is driven by the Hessian of the objective functions. Therefore, for solving optimization problems, the research directions have focused on investigating second-order dynamical systems with Hessian-driven damping.
Nowadays, there have been a wide variety of works devoted to second-order dynamical systems with Hessian-driven damping for finding the solution set of the  unconstrained convex optimization problem:
\begin{equation}\label{unconstrained}
\min_{x\in \mathbb{R}^n} f(x),
\end{equation}
where $ f : \mathbb{R}^n \to \mathbb{R} $ is a twice continuously differentiable convex function.

In order to solve problem (\ref{unconstrained}),  Alvarez et al. \cite{Hessian} first introduce the second-order dynamical system  with  Hessian-driven damping:
\begin{equation}\label{alvar}
\ddot{x}(t)+\alpha\dot{x}(t)+\gamma\nabla^2f(x(t))\dot{x}(t)+\nabla f(x(t))=0,
\end{equation}
 where $\alpha$ and $\gamma$ are positive parameters.  This system (\ref{alvar}) involves a constant viscous damping coefficient $\alpha$  and a constant Hessian-driven damping coefficient $\gamma$, which effectively reduce the oscillations associated with the  Heavy-ball  with
 friction system.
 Subsequently,  Attouch et al. \cite{Attouch2016F} propose  a more general second-order dynamical    system as follows:
\begin{equation}\label{Attouch2016F}
\ddot{x}(t)+\frac{\alpha}{t}\dot{x}(t)+\gamma\nabla^2f(x(t))\dot{x}(t)+\nabla f(x(t))=0,
\end{equation}
where  $\frac{\alpha}{t}$
is the  asymptotic vanishing viscous damping coefficient, and $ \gamma>0 $ is the constant  Hessian-driven  damping coefficient. They show that, as time $t$ approaches infinity, the fast convergence rate of the objective function value along the trajectory is  $ \mathcal{O}\left(\frac{1}{t^2}\right) $.
 To further improve the  convergence rate, by using a time scaling technique, Attouch et al. \cite{Attouch2022F} propose the following second-order dynamical system  with  a more general  Hessian-driven damping coefficient:
\begin{equation}\label{gen}
\ddot{x}(t)+\frac{\alpha}{t}\dot{x}(t)+\delta(t)\nabla^2f(x(t))\dot{x}(t)+\xi(t)\nabla f(x(t))=0,
\end{equation}
where $\delta:\left[t_0,+\infty \right) \to \left[0,+\infty \right) $ is the  Hessian-driven  damping function and $ \xi:\left[t_0,+\infty \right) \to \left[0,+\infty \right) $ is the time scaling function. They show that the convergence rate of the objective function value along the trajectory is $\mathcal{O}\left(\frac{1}{t^2\xi(t)}\right) $.

In the quest for obtaining strong convergence results of solution trajectories, Bo\c{t} et al. \cite{Bot2021T} consider the following Tikhonov regularized second-order dynamical system with asymptotic vanishing viscous damping and  Hessian-driven damping:
 \begin{equation}\label{Bot2021T}
 \ddot{x}(t)+\frac{\alpha}{t}\dot{x}(t)+\gamma\nabla^2f(x(t))\dot{x}(t)+\nabla f(x(t))+\epsilon(t)x(t)=0,
 \end{equation}
where $ \epsilon: \left[t_0,+\infty \right) \to \left[0,+\infty \right) $ is the Tikhonov regularization function. They   show that it inherits the fast convergence properties of the dynamical system (\ref{Attouch2016F}).  Furthermore, they  establish  the strong convergence of trajectory generated by the dynamical system (\ref{Bot2021T}) to the minimum norm solution  of problem (\ref{unconstrained}). Moreover, inspired by the fast convergence rates attached to the heavy ball method in the strongly convex case \cite{Polyak1},  Attouch et al. \cite{Attouch2023A}   propose the following  Tikhonov regularized second-order dynamical system where the viscous damping coefficient is proportional to the square
root of the Tikhonov regularization parameter:
 \begin{equation}\label{att23}
 \ddot{x}(t)+\zeta \sqrt{\epsilon(t)}\dot{x}(t)+\gamma\nabla^2f(x(t))\dot{x}(t)+\nabla f(x(t))+\epsilon(t)x(t)=0,
 \end{equation}
where $\zeta>0$. They establish the fast convergence rate of the objective function value along the trajectory generated by the system (\ref{att23}), and show that the trajectory converges strongly to the minimal norm solution of problem (\ref{unconstrained}).  For a bibliographical review of the development of Tikhonov regularized second-order dynamical systems with Hessian-driven damping for unconstrained optimization problem  (\ref{unconstrained}), we refer the readers to the survey paper \cite{jems23}  and some recent papers \cite{alec21,bagy23,cse24,L2024S,Zhong2024F} for references.

It is important to mention that  all the aforementioned papers concentrate on studying the unconstrained optimization problem (\ref{unconstrained}).
Recently, more researchers have extended second-order  primal-dual dynamical systems from unconstrained optimization problem (\ref{unconstrained}) to the following  linearly constrainted convex optimization problem:
\begin{eqnarray}\label{constrained}
\begin{cases}
\mathop{\mbox{min}}\limits_{x\in\mathbb{R}^n}~~{f(x)}\\
\mbox{s.t.}~~Kx=b,
\end{cases}
\end{eqnarray}
where $K:\mathbb{R}^n\rightarrow\mathbb{R}^m$ is a linear operator and $b\in\mathbb{R}^m$.
 In order to solve problem (\ref{constrained}),  many successful treatments of second-order  primal-dual dynamical systems with viscous damping,  time scaling, and extrapolation coefficients  have been obtained from several different perspectives, see, e.g.  \cite{Chbani2024O,Bot2021I,He2021C,He2022S,heaa23,Hulett2023T,Zeng,2024zhu}.

 However, in contrast to   unconstrained optimization problem (\ref{unconstrained}), we observe that there exist only  few papers devoting to the investigation of second-order  primal-dual dynamical systems with  Hessian-driven damping for solving problem (\ref{constrained}).  He et al. \cite{He2023C} propose the following second-order plus first-order  primal-dual dynamical system with general Hessian-driven damping:
\begin{equation}\label{hemix}
\begin{cases}\ddot{x}(t)+\frac{\alpha}{t}\dot{x}(t)+\delta(t)\frac{d}{dt}\big(\nabla_{x}\mathcal{\widehat{L}}\left(x(t),\lambda(t)\right)\big)+\xi(t)\nabla_{x}\mathcal{\widehat{L}}\left(x(t),\lambda(t)\right)=0,\\
\dot{\lambda}(t)-\eta(t)\nabla_{\lambda} \mathcal{\widehat{L}}\left(x(t)+\frac{t}{\alpha-1}\dot{x}(t),\lambda(t)\right)=0,\end{cases}
\end{equation}
where $ \alpha>1 $, $\frac{\alpha}{t}$
is the  asymptotic vanishing viscous damping coefficient, $\delta:\left[t_0,+\infty \right) \to \left[0,+\infty \right) $ is the  Hessian-driven  damping function,  $ \xi, \eta:\left[t_0,+\infty \right) \to \left[0,+\infty \right) $ are the time scaling functions, and $ \frac{t}{\alpha-1} $ is the extrapolation parameter. Here, $\mathcal{\widehat{L}}$  is the known Lagrangian function defined by
$$\mathcal{\widehat{L}}(x,\lambda):= f(x)+\langle \lambda, Kx-b\rangle.$$
They show that, as time $t$ approaches infinity, the fast convergence rates of the
Lagrangian residual, the objective residual and the feasibility violation along the trajectory  generated by the dynamical system (\ref{hemix}) are $\mathcal{O}\left(\frac{1}{t \eta(t)}\right)$. More recently,
Csetnek and L\'{a}szl\'{o} \cite{cl24} associate to (\ref{constrained}) the Tikhonov regularized second-order primal-dual dynamical system:
\begin{equation}\label{csela}\resizebox{0.93\hsize}{!}{$
\small{\begin{cases}\ddot{x}(t)+\frac{\alpha}{t^q}\dot{x}(t)+\beta t^q\frac{d}{dt}\left(\nabla_{x}\mathcal{\widehat{L}}\left(x(t),y(t)\right)\right)+(1-\frac{\gamma}{t^s})\nabla_{x}\mathcal{\widehat{L}}\left(x(t),y(t)\right)+\frac{c}{t^{2q+s}}{x}(t)=0,\\
 \ddot{y}(t)+\frac{\alpha}{t^q}\dot{y}(t)-\beta t^q\frac{d}{dt}\left(\nabla_y \mathcal{\widehat{L}}\left(x(t),y(t)\right)\right)-(1-\frac{\gamma}{t^s})\nabla_y\mathcal{\widehat{L}}\left(x(t),y(t)\right)+\frac{c}{t^{2q+s}}{y}(t)=0,\end{cases}} $}
\end{equation}
where $ \alpha>1 $, $q,$ $\beta$, $\gamma$, $s$, $c>0$ and $0<q+s<1$. Under mild conditions on the parameters involved, they obtain strong convergence of the trajectories to the minimal norm primal-dual solution, as well as   fast convergence rates for the feasibility measure,  velocities and  objective function values  along the trajectory generated
by the dynamical system (\ref{csela}).

On the other hand,  the problem (\ref{constrained}) is closely related to the convex-concave saddle-point problem, it is therefore natural to employ the  second-order primal-dual dynamical system approaches  to solve   convex-concave saddle point problems. However, it is worth mentioning that there is no   research on Tikhonov regularized second-order primal-dual dynamical systems with Hessian-driven damping for solving  convex-concave saddle point problems although there are some preliminary results available
in \cite{zengifac,ding,2024he,Luo2024AC,ours} on some characterizations of second-order primal-dual dynamical systems with viscous damping, time scaling
and extrapolation coefficients.  This is also the main motivation for the study of second-order primal-dual dynamical systems with Hessian-driven damping and Tikhonov regularization  for solving    convex-concave saddle point problems.

 Motivated by the works reported in \cite{L2024S,ding,2024he}, in this paper, we consider  the following bilinearly coupled convex-concave saddle point problem:
 \begin{equation} \label{PD}
\min_{x\in \mathbb{R}^n} \max_{y\in \mathbb{R}^m}\mathcal{L}(x,y) := f(x)+\left \langle Kx,y \right \rangle-g(y),
\end{equation}
where  $ K:\mathbb{R}^n \to \mathbb{R}^m $ is a linear operator, $\langle\cdot,\cdot\rangle$ represents the standard inner
product of vectors, and both $ f : \mathbb{R}^n \to \mathbb{R} $ and $g : \mathbb{R}^m \to \mathbb{R}$  are twice continuously differentiable convex functions.
For  problem $(\ref{PD})$, we introduce the  following Tikhonov regularized second-order primal-dual dynamical system,  which consists of slow  viscous damping, extrapolation, Hessian-driven damping and time scaling,
\begin{equation}\label{dyn}\resizebox{0.925\hsize}{!}{$
\small{\begin{cases}\ddot{x}(t)+\frac{\alpha}{t^q}\dot{x}(t)+\gamma\frac{d}{dt}\left(\nabla_{x}\mathcal{L}_t\left(x(t),y(t)+\theta(t)\dot{y}(t)\right)\right)+t^s\nabla_{x}\mathcal{L}_t\left(x(t),y(t)+\theta(t)\dot{y}(t)\right)=0,\\
 \ddot{y}(t)+\frac{\alpha}{t^q}\dot{y}(t)-\gamma\frac{d}{dt}\left(\nabla_y \mathcal{L}_t\left(x(t)+\theta (t)\dot{x}(t),y(t)\right)\right)-t^s\nabla_y \mathcal{L}_t\left(x(t)+\theta (t)\dot{x}(t),y(t)\right)=0,\end{cases}} $}
\end{equation}
where $ t\geq t_0>0 $, $\alpha>1$, $0<q<1$, $\gamma>0$ and $ s>0 $, $ \mathcal{L}_t : \mathbb{R}^n \times \mathbb{R}^m \to \mathbb{R}$ is the augmented Lagrangian saddle function (see  (\ref{asd})  for details), $\frac{\alpha}{t^q}  $  is the slow viscous damping parameter, $\gamma$ is the  Hessian-driven damping parameter, $ t^s $ is the time scaling parameter, and  $ \theta(t) $ is the extrapolation parameter defined as
\begin{equation}\label{theta}
  \theta(t):=\frac{t^{2q+s}-2\gamma qt^{2q-1}+\gamma t^q}{(\alpha-1)(t^{q+s}-\gamma q t^{q-1})} .
\end{equation}
We assume that the initial time $ t_0\textgreater\sqrt[s+1]{\gamma q} $, which ensures  $ \theta(t) $ is properly defined for $ t\geq t_0$. The contributions of this paper can be more specifically stated as follows:
\begin{itemize}
\item[{\rm (i)}] We propose a new second-order primal-dual dynamical system (\ref{dyn}) with Hessian-driven damping and  Tikhonov regularization for solving convex-concave saddle point problems. In comparison to the dynamical systems introduced in \cite{2024he,ding,ours},  the  system (\ref{dyn})  incorporates slow  viscous damping, extrapolation, time scaling,  Hessian-driven damping and Tikhonov regularization.
 \item[{\rm (ii)}]  Under mild assumptions on the parameters, we show that the primal-dual gap along the trajectories generated by the dynamical system (\ref{dyn})  enjoys $ \mathcal{O}\left(\frac{1}{t^{2q+s}}\right) $ convergence rate.

\item[{\rm (iii)}] By appropriately adjusting these parameters,  we show that there exists a single regime where both the   convergence rate  of the primal-dual gap and the strong convergence of trajectories  towards  the element of minimal norm of the saddle point set can be simultaneously ensured.

\item[{\rm (iv)}] Through numerical experiments, we demonstrate that  Hessian-driven damping can significantly neutralize the oscillations that occur during iterations.
\end{itemize}

The rest of this paper is organized as follows. In Section 2,  we recall some basic notations and present some preliminary results. In Section 3, we  obtain  the fast convergence properties of the primal-dual gap, and give some integral estimate results.  In Section 4, we establish a simultaneous result concerning the  convergence rate of the augmented primal-dual gap and  the  strong convergence of  the  trajectory  generated by the dynamical system (\ref{dyn}) to the minimal norm solution of the problem (\ref{PD}). In Section 5, we conduct some numerical experiments to illustrate the  theoretical results.

\section{Preliminaries}
Throughout this paper,   $ \mathbb{R}^n$  is the $ n $-dimensional Euclidean space
with inner product and Euclidean norm  denoted by $ \langle \cdot ,\cdot \rangle $ and $ \| \cdot \| $, respectively.
The norm of the Cartesian
 product $ \mathbb{R}^n\times \mathbb{R}^m $ is defined as
\begin{equation*}
\|(x,y)\|=\sqrt{\|x\|^2+\|y\|^2}, ~~~  \forall x\in \mathbb{R}^n  ~\mbox{and}~ y\in  \mathbb{R}^m.
\end{equation*}
For the linear operator $ K:\mathbb{R}^n  \to \mathbb{R}^m $. We denote that
  \begin{equation*}
  \|K\|:=\max\{\|Kx\|:x\in \mathbb{R}^n \mbox{ with } \|x\|\leq1\},
  \end{equation*}
 and $ K^* : \mathbb{R}^m  \to \mathbb{R}^n$ is the adjoint operator of $K$. For a set  $ D\subseteq \mathbb{R}^n $,  $ \mbox{Proj}_{D}0 $ denotes the set of points in $ D $ that are closest to the origin, where $ \mbox{Proj}$ is the projection operator. If $ D $ is a closed convex set, $ \mbox{Proj}_{D}0 $ represents the one with minimal norm in $ D $.

Let $h:\mathbb{R}^n\to \mathbb{R}$ be a continuously differentiable convex  function. We say that $ h $ is an $ \epsilon $-strong  convex function with a strong convexity parameter $ \epsilon\geq 0 $ iff $h-\frac{\epsilon}{2}\|\cdot\|^2  $ is a convex function. Clearly,  
   \begin{equation}\label{strong}
   \left\langle \nabla h(x_1),x_2-x_1 \right\rangle \leq h(x_2)-h(x_1)-\frac{\epsilon}{2}\|x_1-x_2\|^2, ~~~ \forall x_1, x_2\in \mathbb{R}^n.
   \end{equation}

Now, consider the saddle point problem  (\ref{PD}). A pair $ (x^*,y^*) \in \mathbb{R}^n\times \mathbb{R}^m$ is said to be a saddle point of the  Lagrangian function $ \mathcal{L} $ iff
 \begin{equation}\label{Saddlepoint}
 \mathcal{L}(x^*,y)\leq \mathcal{L}(x^*,y^*) \leq \mathcal{L}(x,y^*),~~\forall (x,y)\in\mathbb{R}^n\times\mathbb{R}^m.
 \end{equation}
In the sequel,   the set of saddle points of $\mathcal{L}$ is denoted by  $ \mathrm{\Omega} $.  Moreover, we assume that $\mathrm{\Omega}\not=\emptyset$.

For $ c>0 $ and $ p>0 $, associated with the Lagrangian function $ \mathcal{L} $, we introduce   an  augmented  Lagrangian function $ \mathcal{L}_t: \mathbb{R}^n \times \mathbb{R}^m \to \mathbb{R}$ defined as
\begin{equation}\label{asd}
\begin{split}
\mathcal{L}_t(x,y)&:= \mathcal{L}(x,y)  + \frac{c}{2t^p}\left(\|x\|^2-\|y\|^2\right)\\
&=f(x)+\left \langle Kx,y \right \rangle-g(y)+\frac{c}{2t^p}\left(\|x\|^2-\|y\|^2\right). 
\end{split}
\end{equation}
Obviously, $ \mathcal{L}_t(\cdot,y)$ is $ \frac{c}{t^p} $-strong convex for any $ y\in \mathbb{R}^m $, and $ \mathcal{L}_t(x,\cdot)$ is $ \frac{c}{t^p} $-strong concave for any $ x\in \mathbb{R}^n $. This means that $ \mathcal{L}_t$ admits a unique  saddle point $ (x_t,y_t) \in \mathbb{R}^n \times \mathbb{R}^m$, i.e.,
 \begin{equation}\label{Saddlepoint-t}
 \mathcal{L}_t(x_t,y)\leq \mathcal{L}_t(x_t,y_t) \leq \mathcal{L}_t(x,y_t),~~\forall (x,y)\in\mathbb{R}^n\times\mathbb{R}^m.
 \end{equation}
 Naturally, the system of primal-dual optimality conditions reads
 \begin{equation}\label{KKT}
 \begin{cases}
 0=\nabla_x\mathcal{L}_t(x_t,y_t) =\nabla  f(x_t)+\frac{c}{t^p}x_t+K^*y_t,\\
 0=\nabla_y\mathcal{L}_t(x_t,y_t)=-\nabla g(y_t)-\frac{c}{t^p}y_t+Kx_t.
 \end{cases}
 \end{equation}
Here, $\nabla _x\mathcal{L}_t$ and $\nabla _y\mathcal{L}_t$ denote the gradients of $\mathcal{L}_t$ with respect to the first argument and  the second argument, respectively.

The following important property will be used in the sequel.
\begin{lemma}\textup{\cite[Lemma 2.3]{Chbani2024O}}   \label{xtyt}
Let $ (\bar{x}^*,\bar{y}^*) \in\mathrm{Proj}_{\Omega}0 $ and let $(x_t,y_t) $ be the saddle point of $ \mathcal{L}_t $. Suppose $ 0<p<1 $. Then,
\begin{itemize}
\item[{\rm (i)}]   $ \lim_{t\to +\infty} \|(x_t,y_t)-(\bar{x}^*,\bar{y}^*)\| =0$ and $ \|(x_t,y_t)\|\leq\|(\bar{x}^*,\bar{y}^*)\| $, for all $ t $ $\geq$ $ t_0 $.
\item[{\rm (ii)}]   $ \|(\dot{x}_t,\dot{y}_t)\|\leq \frac{p}{t}\|(x_t,y_t) \|\leq \frac{p}{t}\|(\bar{x}^*,\bar{y}^*) \| $, for all $ t $ $\geq$ $ t_0 $.
\end{itemize}
\end{lemma}
\begin{remark}
Let $ z_t=(x_t,y_t) $ and $ \bar{z}^*=(\bar{x}^*,\bar{y}^*) $. By Lemma \ref{xtyt}, it is easy to see that
\begin{equation}\label{z1}
\max\left\{\|x_t\|^2,\|y_t\|^2\right\}\leq \|\bar{z}^*\|^2 
~~\mbox{and}~~
\max\left\{\|\dot{x}_t\|^2,\|\dot{y}_t\|^2\right\}\leq \|\dot{z}_t\|^2\leq \frac{p^2}{t^2} \|\bar{z}^*\|^2.
\end{equation}
\end{remark}

\section{Fast convergence of the values}
In this section, by using the tools and techniques provided by the Lyapunov analysis, we  establish  the fast convergence rate  of  primal-dual gap along the trajectory generated by the dynamical system \textup{(\ref{dyn})}. Moreover, we also give some integral estimates.

\begin{theorem}\label{Th1}
Let $ (x ,y ):\left[t_0, +\infty\right)\to\mathbb{R}^n \times \mathbb{R}^m  $ be a global solution of the dynamical system $ (\ref{dyn}) $. Suppose that $ s>\max\{0, p-2\} $ and $p-q-s-1>0 $.
Then, for any $ (x^*,y^*)\in \Omega $, the trajectory $ (x(t),y(t))_{t\geq t_0} $ is bounded and
\begin{eqnarray*}
\mathcal{L}(x(t),y^*)-\mathcal{L}(x^*,y(t))=\mathcal{O}\left(\frac{1}{t^{2q+s}}\right),~~~as ~  t\to +\infty ,~~~~~~~~~~\\
\|\dot{x}(t)+\gamma \nabla_{x}\mathcal{L}_t\left(x(t),y(t)+\theta(t)\dot{y}(t)\right)\|=\mathcal{O}\left( \frac{1}{t^q}\right),~~~as ~  t\to +\infty ,~~~~~~~~~~\\
\|\dot{y}(t)-\gamma\nabla_{y}\mathcal{L}_t\left(x(t)+\theta(t)\dot{x}(t),y(t)\right)\|=\mathcal{O}\left( \frac{1}{t^q}\right),~~~as ~  t\to +\infty ,~~~~~~~~~~\\
\int_{t_0}^{+\infty}t^q\left( \|\dot{x}(t)\|^2+   \|\dot{y}(t)\|^2        \right)dt< +\infty,~~~~~~~~~~~~~~~~~~~ \\
\int_{t_0}^{+\infty} t^{2q+s}\left(\|\nabla_x\mathcal{L}_t(x(t),y(t)+\theta(t)\dot{y}(t))\|^2+\|\nabla_y\mathcal{L}_t(x(t)+\theta(t)\dot{x}(t),y(t))\|^2\right)dt  < +\infty.
\end{eqnarray*}
\end{theorem}
\begin{proof}

 For any fixed $ (x^*,y^*)\in \mathrm{\Omega} $, we introduce the energy function $ \mathcal{E}:\left[t_0,+\infty \right) \to \mathbb{R} $ as follows:
\begin{equation}\label{def1}
\mathcal{E}(t)=\mathcal{E}_1(t)+\mathcal{E}_2(t)+\mathcal{E}_3(t),
\end{equation}
where
\begin{equation*}
\small{\begin{cases}
\mathcal{E}_1(t)=&\left(t^{2q+s}-2\gamma q t^{2q-1}+\gamma t^q\right)
\left(\mathcal{L}(x(t),y^*)-\mathcal{L}(x^*,y(t))+\frac{c}{2t^p}(\|x(t)\|^2+\|y(t)\|^2)\right),\\
\mathcal{E}_2(t)=&\frac{1}{2}\lVert(\alpha-1)(x(t)-x^*)+t^q \big(\dot{x}(t)+\gamma \nabla_{x}\mathcal{L}_t\left(x(t),y(t)+\theta(t)\dot{y}(t)\right)\big)\rVert^2\\
                 &+\frac{\alpha-1}{2}(1-qt^{q-1})\|x(t)-x^*\|^2,\\
\mathcal{E}_3(t)=&\frac{1}{2}\|(\alpha-1)(y(t)-y^*)+t^q \big(\dot{y}(t)-\gamma\nabla_{y}\mathcal{L}_t\left(x(t)+\theta(t)\dot{x}(t),y(t)\right)\big)\|^2\\
                 &+\frac{\alpha-1}{2}(1-qt^{q-1})\|y(t)-y^*\|^2.
\end{cases}}
\end{equation*}
By $ 0<q<1 $ and $ s>0 $, there exists $ t_1\geq t_0 $ such that $t^{2q+s}-2\gamma q t^{2q-1}+\gamma t^q\geq0$ and $ 1-qt^{q-1}\geq0 $,  for all $ t\geq t_1 $. Thus, $ \mathcal{E}(t) \geq0  $, for all $ t\geq t_1 $.

   We now analyze the time derivative of $ \mathcal{E}(t) $.
   
    Clearly,
   \begin{equation}\label{5}\small{
\begin{split}
\dot{\mathcal{E}}_1(t)
=&\big((2q+s)t^{2q+s-1}-2\gamma q(2q-1)t^{2q-2}+\gamma qt^{q-1}\big)\Big(\mathcal{L}(x(t),y^*)-\mathcal{L}(x^*,y(t))\\
&+\frac{c}{2t^p} ( \| x(t) \|^2+ \| y(t) \|^2)\Big)+\left(t^{2q+s}-2\gamma q t^{2q-1}+\gamma t^q\right) \Big(\left \langle \nabla_x\mathcal{L}(x(t), y^*),\dot{x}(t)\right \rangle \\
&-\left \langle \nabla_y\mathcal{L}(x^*, y(t)),\dot{y}(t)\right \rangle-\frac{cp}{2t^{p+1}} \left( \| x(t) \|^2 +\|y(t)\|^2\right)\\ &+\frac{c}{t^p}(\left \langle x(t),\dot{x}(t)\right \rangle+\left \langle y(t),\dot{y}(t)\right \rangle)\Big).
\end{split}}
\end{equation}

Now, set $ \mu(t):=(\alpha-1)(x(t)-x^*)+t^q \Big(\dot{x}(t)+\gamma \nabla_{x}\mathcal{L}_t\big(x(t),y(t)+\theta(t)\dot{y}(t)\big)\Big). $
Then,
\begin{align*}
\dot{\mu}(t)=&(\alpha-1)\dot{x}(t)+qt^{q-1}\Big(\dot{x}(t)+\gamma \nabla_{x}\mathcal{L}_t\big(x(t),y(t)+\theta(t)\dot{y}(t)\big)\Big)\\
&+t^q\Big(\ddot{x}(t)+\gamma\frac{d}{dt} \nabla_{x}\mathcal{L}_t\big(x(t),y(t)+\theta(t)\dot{y}(t)\big)\Big).
\end{align*}
This, together with the first equality of   (\ref{dyn}), follows that
\begin{equation}\label{dot-mu}
\dot{\mu}(t)=(qt^{q-1}-1)\dot{x}(t)+(\gamma qt^{q-1}-t^{q+s})\nabla_{x}\mathcal{L}_t\big(x(t),y(t)+\theta (t)\dot{y}(t)\big).
\end{equation}
Note that
$$
\nabla_{x}\mathcal{L}_t\left(x(t),y(t)+\theta(t)\dot{y}(t)\right)=\nabla_{x}\mathcal{L}(x(t),y^*)+K^*\left(y(t)-y^*+\theta(t)\dot{y}(t)\right)+\frac{c}{t^p}x(t).
$$
Therefore,
 \begin{equation}\label{mu000}
\begin{split}
\left \langle \mu(t),\dot{\mu}(t)\right \rangle=&(\alpha-1)(qt^{q-1}-1)\left \langle x(t)-x^*,\dot{x}(t)\right \rangle+t^q(qt^{q-1}-1)\| \dot{x}(t)\|^2 \\
&+(\alpha-1)(\gamma qt^{q-1}-t^{q+s})\left \langle\nabla_x\mathcal{L}(x(t),y^*)+\frac{c}{t^p}x(t),x(t)-x^*\right \rangle \\
&+(\alpha-1)(\gamma qt^{q-1}-t^{q+s})\left \langle K^*(y(t)-y^*+\theta(t)\dot{y}(t)),x(t)-x^*\right \rangle \\
&-(t^{2q+s}-2\gamma qt^{2q-1}+\gamma t^q) \left \langle\nabla_x\mathcal{L}(x(t),y^*)+\frac{c}{t^p}x(t),\dot{x}(t)\right \rangle \\
&-(t^{2q+s}-2\gamma qt^{2q-1}+\gamma t^q)\left \langle K^*(y(t)-y^*+\theta(t)\dot{y}(t)),\dot{x}(t)\right \rangle \\
&+\gamma t^q(\gamma qt^{q-1}-t^{q+s})\|\nabla_x\mathcal{L}_t(x(t),y(t)+\theta(t)\dot{y}(t))\|^2. 
\end{split}
\end{equation}
Since $\mathcal{L}(\cdot,y^*)+\frac{c}{2t^p}\|\cdot\|^2 $ is an $\epsilon$-strong convex function with  strong convexity parameter $\epsilon=\frac{c}{t^p}$, it follows from $(\ref{strong}) $ that
\begin{equation}\label{stre}
\begin{aligned}
&\left \langle\nabla_x\mathcal{L}(x(t),y^*)+\frac{c}{t^p}x(t),x^*- {x}(t)\right \rangle \\ \leq&\mathcal{L}(x^*,y^*)-\mathcal{L}(x(t),y^*)+\frac{c}{2t^p}\left(\|x^*\|^2- \|x(t)\|^2- \|x^*-x(t)\|^2\right).
\end{aligned}
\end{equation}
Note that $ t_0\textgreater\sqrt[s+1]{\gamma q} $, which implies
\begin{equation}\label{stresub}
\gamma qt^{q-1}-t^{q+s}< 0,~~~ \forall  t\geq t_0.
\end{equation}
Combining $(\ref{mu000}) $,  $(\ref{stre}) $ and $(\ref{stresub}) $,  we have 
\begin{equation}\label{22-22}\resizebox{0.925\hsize}{!}{$
\begin{split}
&\left \langle \mu(t),\dot{\mu}(t)\right \rangle\\
\leq&(\alpha-1)(qt^{q-1}-1)\left \langle x(t)-x^*,\dot{x}(t)\right \rangle+t^q(qt^{q-1}-1)\| \dot{x}(t)\|^2 \\
&+(\alpha-1)(\gamma qt^{q-1}-t^{q+s})\Big(\mathcal{L}(x(t),y^*)-\mathcal{L}(x^*,y^*)+\frac{c}{2t^p}\Big(\|x(t)\|^2-\|x^*\|^2 \\
&+\|x(t)-x^*\|^2\Big)\Big) +(\alpha-1)(\gamma qt^{q-1}-t^{q+s})\left \langle K^*(y(t)-y^*+\theta(t)\dot{y}(t)),x(t)-x^*\right \rangle\\
&-(t^{2q+s}-2\gamma qt^{2q-1}+\gamma t^q)\left \langle\nabla_x\mathcal{L}(x(t),y^*)+\frac{c}{t^p}x(t),\dot{x}(t)\right \rangle\\
&-(t^{2q+s}-2\gamma qt^{2q-1}+\gamma t^q)\left \langle K^*(y(t)-y^*+\theta(t)\dot{y}(t)),\dot{x}(t)\right \rangle\\
&+\gamma t^q(\gamma qt^{q-1}-t^{q+s})\|\nabla_x\mathcal{L}_t(x(t),y(t)+\theta(t)\dot{y}(t))\|^2,    \end{split}  $}
\end{equation}
  for all $t\geq t_1$. Moreover,
\begin{equation}\label{e2d}
\begin{aligned}
\frac{d}{dt}\Big( \frac{\alpha-1}{2}(1-qt^{q-1})\|x(t)-x^*\|^2 \Big)
=&\frac{(\alpha-1)q(1-q)}{2}t^{q-2}\|x(t)-x^*\|^2\\
&+(\alpha-1)(1-qt^{q-1})\left \langle x(t)-x^*,\dot{x}(t)\right \rangle.
\end{aligned}
\end{equation}
By (\ref{22-22}) and (\ref{e2d}),    we obtain 
 \begin{equation}\label{6}
\begin{split}
\dot{\mathcal{E}}_2(t)=&\left \langle \mu(t),\dot{\mu}(t)\right \rangle +\frac{d}{dt}\Big( \frac{\alpha-1}{2}(1-qt^{q-1})\|x(t)-x^*\|^2 \Big)\\
\leq&t^q(qt^{q-1}-1)\| \dot{x}(t)\|^2+(\alpha-1)(\gamma qt^{q-1}-t^{q+s})\Big(\mathcal{L}(x(t),y^*)\\
&-\mathcal{L}(x^*,y^*)+\frac{c}{2t^p} \left(\|x(t)\|^2-\|x^*\|^2\right)\Big)+\Big(\frac{(\alpha-1)q(1-q)}{2}t^{q-2}\\
&+\frac{(\alpha-1)c}{2}(\gamma qt^{q-1-p}-t^{q+s-p})\Big)         \|x(t)-x^*\|^2   \\
&+(\alpha-1)(\gamma qt^{q-1}-t^{q+s})\left \langle K^*(y(t)-y^*+\theta(t)\dot{y}(t)),x(t)-x^*\right \rangle\\
&-(t^{2q+s}-2\gamma qt^{2q-1}+\gamma t^q)\left \langle\nabla_x\mathcal{L}(x(t),y^*)+\frac{c}{t^p}x(t),\dot{x}(t)\right \rangle\\
&-(t^{2q+s}-2\gamma qt^{2q-1}+\gamma t^q)\left \langle K^*(y(t)-y^*+\theta(t)\dot{y}(t)),\dot{x}(t)\right \rangle\\
&+\gamma t^q(\gamma qt^{q-1}-t^{q+s})\|\nabla_x\mathcal{L}_t(x(t),y(t)+\theta(t)\dot{y}(t))\|^2,
\end{split}
\end{equation}
 for all $t\geq t_1$.

 Similarly, for all $t\geq t_1$,
\begin{equation}\label{7}
\begin{split}
\dot{\mathcal{E}}_3(t)\leq &t^q(qt^{q-1}-1)\| \dot{y}(t)\|^2+(\alpha-1)(\gamma qt^{q-1}-t^{q+s})\Big(\mathcal{L}(x^*,y^*)\\
&-\mathcal{L}(x^*,y(t))+\frac{c}{2t^p} \left(\|y(t)\|^2-\|y^*\|^2\right)\Big)+\Big(\frac{(\alpha-1)q(1-q)}{2}t^{q-2}\\
&+\frac{(\alpha-1)c}{2}(\gamma qt^{q-1-p}-t^{q+s-p})\Big)         \|y(t)-y^*\|^2   \\
&-(\alpha-1)(\gamma qt^{q-1}-t^{q+s})\left \langle K(x(t)-x^*+\theta(t)\dot{x}(t)),y(t)-y^*\right \rangle\\
&+(t^{2q+s}-2\gamma qt^{2q-1}+\gamma t^q) \left \langle\nabla_y\mathcal{L}(x^*,y(t))-\frac{c}{t^p}y(t),\dot{y}(t)\right \rangle  \\
&+(t^{2q+s}-2\gamma qt^{2q-1}+\gamma t^q) \left \langle K(x(t)-x^*+\theta(t)\dot{x}(t)),\dot{y}(t)\right \rangle\\
&+\gamma t^q(\gamma qt^{q-1}-t^{q+s})\|\nabla_y\mathcal{L}_t(x(t)+\theta(t)\dot{x}(t),y(t))\|^2.
\end{split}
\end{equation}
Note that, by  $(\ref{theta}) $, we have
\begin{align*}
&(\alpha-1)(\gamma qt^{q-1}-t^{q+s})\left \langle K^*(y(t)-y^*+\theta(t)\dot{y}(t)),x(t)-x^*\right \rangle-(t^{2q+s}-2\gamma qt^{2q-1}&\\
&+\gamma t^q)\left \langle K^*(y(t)-y^*+\theta(t)\dot{y}(t)),\dot{x}(t)\right \rangle-(\alpha-1)(\gamma qt^{q-1}-t^{q+s})\left \langle K(x(t)-x^*\right.&\\
& \left.+\theta(t)\dot{x}(t)),y(t)-y^*\right \rangle
+(t^{2q+s}-2\gamma qt^{2q-1}+\gamma t^q) \left \langle K(x(t)-x^*+\theta(t)\dot{x}(t)),\dot{y}(t)\right \rangle=0.&
\end{align*}
Therefore, combining   $(\ref{5})$, $(\ref{6})$ and $(\ref{7})$,  we derive that  
\begin{equation}\label{dt}\resizebox{0.925\hsize}{!}{$
\begin{aligned}
\dot{\mathcal{E}}(t)\leq&a(t)(\mathcal{L}(x(t),y^*)-\mathcal{L}(x^*,y(t)))+b(t)\left(\|x(t)\|^2+\|y(t)\|^2\right) \\
&+\frac{\alpha-1}{2}\left(q(1-q)t^{q-2}+c\gamma qt^{q-1-p}-ct^{q+s-p}      \right)        ( \|x(t)-x^*\|^2+\|y(t)-y^*\|^2)\\
&+\gamma t^q(\gamma qt^{q-1}-t^{q+s})\Delta(t)+t^q(qt^{q-1}-1)(\| \dot{x}(t)\|^2+\| \dot{y}(t)\|^2)\\
&+\frac{(\alpha-1)c}{2}(t^{q+s-p}-\gamma qt^{q-1-p} )\left(\|x^*\|^2+\|y^*\|^2\right),
\end{aligned}$}
\end{equation}
for all $t\geq t_1$, where
\begin{equation}\label{at}
a(t):= (2q+s)t^{2q+s-1}-2\gamma q(2q-1)t^{2q-2}+\gamma q t^{q-1}+(\alpha-1)(\gamma qt^{q-1}-t^{q+s}),
\end{equation}
\begin{equation*}
\begin{aligned}
b(t):=& \frac{c}{2}\Big((2q+s-p)t^{2q+s-p-1}-2\gamma q(2q-1-p)t^{2q-p-2}+\gamma(q-p)t^{q-p-1} \\
&~~~~~+(\alpha-1)(\gamma qt^{q-p-1}-t^{q+s-p})\Big),
\end{aligned}
\end{equation*}
and
\begin{equation}\label{Delta}
 \Delta(t):= \|\nabla_x\mathcal{L}_t(x(t),y(t)+\theta(t)\dot{y}(t))\|^2+\|\nabla_y\mathcal{L}_t(x(t)+\theta(t)\dot{x}(t),y(t))\|^2.
\end{equation}

Now, we analyze the coefficients on the right hand side of (\ref{dt}).

Firstly, from $ 0<q<1 $ and $ s>0 $, we have
\begin{equation}\label{coef0}
q+s>2q+s-1>q-1>2q-2 .
 \end{equation}
 This, together with $ \alpha>1 $, yields that there exists $ t_2\geq t_1 $ such that
 \begin{equation}\label{coe0}
  a(t)\leq 0 \mbox{ and } b(t)\leq 0, ~~\forall  t\geq t_2.
 \end{equation}
Secondly, by $ s>\max\{0, p-2\} $, there exists  $ t_3\geq t_2 $ such that
\begin{equation}\label{coe00}
\frac{\alpha-1}{2}\left(q(1-q)t^{q-2}+c\gamma qt^{q-1-p}-ct^{q+s-p}\right)\leq 0,~~~\forall~t\geq t_3.
\end{equation}
Thirdly, we can deduce that there exists $ t_4\geq t_3 $ such that
 \begin{equation}\label{coe000}
t^q(qt^{q-1}-1) \leq0, ~~\forall  t\geq t_4.
 \end{equation}
Finally, by (\ref{Saddlepoint}), we have $  \mathcal{L}(x(t),y^*)-\mathcal{L}(x^*,y(t))\geq0 $. This, together with $(\ref{stresub})$,  $(\ref{coe0})$, $(\ref{coe00})$ and $(\ref{coe000})$, follows that
\begin{equation*}\label{int}
\dot{\mathcal{E}}(t)\leq\frac{(\alpha-1)c}{2}t^{q+s-p}\left(\|x^*\|^2+\|y^*\|^2\right), ~~~\forall   t\geq   t_4.
\end{equation*}
Integrating  it from $ t_4 $ to $ t $, we obtain
\begin{equation*}
\mathcal{E}(t)\leq \mathcal{E}(t_4)+\int_{t_4}^{t} \frac{(\alpha-1)c}{2}\omega^{q+s-p} \left(\|x^*\|^2+\|y^*\|^2\right)d\omega, ~~~\forall~ t\geq t_4.
\end{equation*} Since  $p-q-s-1>0 $,    there exists   $ M_1\geq 0 $ such that
\begin{equation*}
  \mathcal{E}(t)   \leq M_1,~~~\forall~ t\geq t_4.
\end{equation*}
Taking  into account (\ref{def1}), we derive that  the trajectory $ (x(t),y(t))  $, $  \|(\alpha-1)(x(t)-x^*)+t^q(\dot{x}(t)+\gamma \nabla_{x}\mathcal{L}_t(x(t),y(t)+\theta(t)\dot{y}(t)))\|  $ and $\|(\alpha-1)(y(t)-y^*)+t^q \left(\dot{y}(t)-\gamma\nabla_{y}\mathcal{L}_t\left(x(t)+\theta(t)\dot{x}(t),y(t)\right)\right)\| $ are bounded for all $ t\geq t_4 $.
Note that
\begin{equation*}
\begin{aligned}
&\|t^q\left(\dot{x}(t)+\gamma \nabla_{x}\mathcal{L}_t\left(x(t),y(t)+\theta(t)\dot{y}(t)\right)\right)\|^2\\
\leq&2\|(\alpha-1)(x(t)-x^*)\|^2+ 2\|(\alpha-1)(x(t)-x^*)+\\
&~~~~~~~~~~~~~~~~~~~~~~~~~~~~~~~~t^q(\dot{x}(t)+\gamma \nabla_{x}\mathcal{L}_t\left(x(t),y(t)+\theta(t)\dot{y}(t)\right))\|^2.
\end{aligned}
\end{equation*}
Then, by $  \|(\alpha-1)(x(t)-x^*)+t^q\left(\dot{x}(t)+\gamma \nabla_{x}\mathcal{L}_t\left(x(t),y(t)+\theta(t)\dot{y}(t)\right)\right)\|  $ and the trajectory $ (x(t),y(t)) $ are  bounded for all $ t\geq t_4 $, we can deduce  that
\begin{equation*}
\|\dot{x}(t)+\gamma \nabla_{x}\mathcal{L}_t\left(x(t),y(t)+\theta(t)\dot{y}(t)\right)\|=\mathcal{O}\left( \frac{1}{t^q}\right), ~~~\mbox{as }~t  \to +\infty.
\end{equation*}
Similarly, we can show that
\begin{equation*}
\|\dot{y}(t)-\gamma\nabla_{y}\mathcal{L}_t\left(x(t)+\theta(t)\dot{x}(t),y(t)\right)\|=\mathcal{O}\left( \frac{1}{t^q}\right), ~~~\mbox{as }~t  \to +\infty.
\end{equation*}
Moreover, by (\ref{def1}), we have
\begin{equation*}\label{Corrolary}
\left(t^{2q+s}-2\gamma q t^{2q-1}+\gamma t^q\right)(\mathcal{L}(x(t),y^*)-\mathcal{L}(x^*,y(t))\leq \mathcal{E}(t) \leq M_1, ~~~\forall ~ t\geq t_4,
\end{equation*}
which implies
\begin{equation*}
\mathcal{L}(x(t),y^*)-\mathcal{L}(x^*,y(t))=\mathcal{O}\left( \frac{1}{t^{2q+s}}\right),~~~\mbox{as } ~ t\to +\infty.
\end{equation*}

On the other hand, it is clear  from (\ref{dt}) that
\begin{equation*}\label{r6}
\begin{aligned}
&\dot{\mathcal{E}}(t)+t^q(1-qt^{q-1})(\| \dot{x}(t)\|^2+\| \dot{y}(t)\|^2)
+\gamma t^q(t^{q+s}-\gamma qt^{q-1})\Delta(t)\\
\leq&\frac{(\alpha-1)c}{2}t^{q+s-p}\left(\|x^*\|^2+\|y^*\|^2\right), ~~~ \forall t\geq t_4 .
\end{aligned}
\end{equation*}
Since $0<q<1 $, there exists  $ t_5\geq t_4 $ such that $$ \frac{1}{2}t^q\leq t^q( 1-qt^{q-1})~\mbox{and}~\frac{1}{2}\gamma t^{2q+s}\leq\gamma t^q(t^{q+s}-\gamma qt^{q-1}),~~\forall~t\geq t_5.$$
 Then, for any  $t \geq   t_5 $,
\begin{equation*}
\dot{\mathcal{E}}(t)+\frac{1}{2}t^q(\| \dot{x}(t)\|^2+\| \dot{y}(t)\|^2)
+\frac{1}{2}\gamma t^{2q+s}\Delta(t)
\leq\frac{(\alpha-1)c}{2}t^{q+s-p} \left(\|x^*\|^2+\|y^*\|^2\right).
\end{equation*} Integrating it from $ t_5$ to $ t $, we have
\begin{equation*}
\begin{aligned}
&\mathcal{E}(t)+\int_{t_5}^{t} \frac{1}{2}\omega^q(\| \dot{x}(\omega)\|^2+\| \dot{y}(\omega)\|^2)d\omega+ \int_{t_5}^{t}\frac{1}{2}\gamma \omega^{2q+s}\Delta(\omega)d\omega\\
&~~~~~ \leq\mathcal{E}(t_5)+ \int_{t_5}^{t} \frac{(\alpha-1)c}{2}\omega^{q+s-p} \left(\|x^*\|^2+\|y^*\|^2\right)d\omega, ~\forall~ t\geq t_5.
\end{aligned}
\end{equation*}
According to $p-q-s-1>0 $ and  $  \mathcal{E}(t)\geq 0 $ for all $ t\geq t_5 $, we have
\begin{equation*}
\int_{t_0}^{+\infty}t^q\left( \|\dot{x}(t)\|^2+   \|\dot{y}(t)\|^2        \right)dt< +\infty
\mbox{ and }
\int_{t_0}^{+\infty}  t^{2q+s}\Delta(t)dt
  < +\infty.
\end{equation*}
The proof is complete.
\end{proof}
\begin{remark}
In \cite[Theorem 4]{L2023O}, L\'{a}szl\'{o} provided a Tikhonov regularized second-order dynamical system with slow vanishing damping for solving the problem (\ref{unconstrained}), and obtained that
\begin{equation*}
f(x(t))-\min f=\mathcal{O}\left(\frac{1}{t^{2q}}\right),~~~\mbox{ and }
\|\dot{x}(t)\|=\mathcal{O}\left( \frac{1}{t^q}\right),~~~as ~  t\to +\infty.
\end{equation*}
However, by taking $K\equiv0$, $g\equiv0$, $\gamma\equiv0$ and $\theta\equiv0$ in the primal-dual dynamical system (\ref{dyn}), it follows from Theorem \ref{Th1} that
\begin{equation*}
f(x(t))-\min f=\mathcal{O}\left(\frac{1}{t^{2q+s}}\right),~~~\mbox{ and }
\|\dot{x}(t)\|=\mathcal{O}\left( \frac{1}{t^q}\right),~~~as ~  t\to +\infty.
\end{equation*}
Therefore, Theorem \ref{Th1} can be viewed as an extension  of  \cite[Theorem 4]{L2023O}, which deals with the unconstrained convex optimization problem (\ref{unconstrained}), to the convex-concave bilinear saddle point problem (\ref{PD}).
\end{remark}

\section{Strong convergence of trajectory to the minimal norm solution}
In this section, by using  Tikhonov regularization techniques, we establish a simultaneous result concerning  both the convergence rate of the primal-dual gap and  the strong convergence of trajectory  generated by dynamical system (\ref{dyn}). In order to do this, we first state the following lemma, which gives the estimates on a newly defined energy function.
\begin{lemma}\label{lemma4.21}
 Let $ (x ,y ):\left[t_0, +\infty\right)\to\mathbb{R}^n \times \mathbb{R}^m  $ be a global solution of the dynamical system \textup{(\ref{dyn})} and let $ (x_t,y_t) $ be the saddle point of $ \mathcal{L}_t $. Consider the energy function $ \hat{\mathcal{E}}:\left[t_0,+\infty \right) \to \mathbb{R} $  defined as
 \begin{equation}\label{def}
 \hat{\mathcal{E}}(t)=\hat{\mathcal{E}}_1(t)+\hat{\mathcal{E}}_2(t)+\hat{\mathcal{E}}_3(t),
 \end{equation}
 where
  \begin{equation*}
 \left\{
 \begin{array}{rl}
 \hat{\mathcal{E}}_1(t)=&\left(t^{2q+s}-2\gamma q t^{2q-1}+\gamma t^q\right)
 \left(\mathcal{L}_t(x(t),y_t)-\mathcal{L}_t(x_t,y(t))\right),\\
 \hat{\mathcal{E}}_2(t)=&\frac{1}{2}\|(\alpha-1)(x(t)-x_t)+t^q \big(\dot{x}(t)+\gamma \nabla_{x}\mathcal{L}_t(x(t),y(t)+\theta(t)\dot{y}(t))\big)\|^2\\
 &+\frac{\alpha-1}{2}(1-qt^{q-1})\|x(t)-x_t\|^2,\\
 \hat{\mathcal{E}}_3(t)=&\frac{1}{2}\|(\alpha-1)(y(t)-y_t)+t^q \big(\dot{y}(t)-\gamma\nabla_{y}\mathcal{L}_t(x(t)+\theta(t)\dot{x}(t),y(t))\big)\|^2\\
 &+\frac{\alpha-1}{2}(1-qt^{q-1})\|y(t)-y_t\|^2.
 \end{array}\right.
 \end{equation*}
 Then, there exists $ t_1\geq t_0 $ such that 
 \begin{equation*}\resizebox{0.95\hsize}{!}{$
\begin{aligned}
&\frac{M}{t^q}\hat{\mathcal{E}}(t)+\dot{\hat{\mathcal{E}}}(t)\\
\leq&\Big(M\left(t^{q+s}-2\gamma q t^{q-1}+\gamma\right)+a(t) \Big)\left(\mathcal{L}_t(x(t),y_t)-\mathcal{L}(x_t,y(t))\right)\\
& -\frac{cp}{2}\left(t^{2q+s-p-1}-2\gamma q t^{2q-p-2}+\gamma t^{q-p-1}\right)\left((\|x(t)\|^2+\|y(t)\|^2)-(\|x_t\|^2+\|y_t\|^2)\right)\\
&+l(t)\left(\|x(t)-x_t\|^2+\|y(t)-y_t\|^2\right)
+\Big(\frac{3M-1}{2}t^{q}+qt^{2q-1}\Big)\left(\| \dot{x}(t)\|^2+\| \dot{y}(t)\|^2\right)\\
&+\Big(\frac{3M}{2}\gamma^2t^{q}+\gamma^2 q t^{2q-1}-\frac{1}{2}\gamma t^{2q+s}\Big)\Delta(t)\\
&+ \Big( \frac{2}{(\alpha-1)c}(\|K\|^2+\|K^*\|^2)\left(t^{3q+s+p}-2\gamma  qt^{3q-1+p}+\gamma t^{2q+p}\right)  \\ &+\frac{2(\alpha-1)\alpha}{c}\left(\alpha t^{p-s-q}-q t^{p-s-1}\right)  +   \frac{1}{2}(\alpha-1)^2t^{q}+   \frac{1}{2}\gamma(\alpha-1)^2  t^{-s}
\Big)  (\|\dot{x}_t\|^2+\|\dot{y}_t\|^2),
\end{aligned}$}
 \end{equation*}
 for all $ t\geq t_1 $. Here, $M$ is an arbitrarily positive constant,
  \begin{equation*}\resizebox{0.92\hsize}{!}{$
\begin{aligned}
l(t):=&\frac{M(\alpha-1)(3\alpha-2)}{2}t^{-q}+\frac{(\alpha-1)c\gamma}{8}t^{-p}-\frac{M(\alpha-1)q}{2}t^{-1}-\frac{(\alpha-1)c}{4}t^{q+s-p}\\
&+\frac{(\alpha-1)c\gamma q }{4}t^{q-1-p}
-\frac{(\alpha-1)qc}{8\alpha}t^{2q+s-1-p}+\frac{(\alpha-1)q(1-q)}{2}t^{q-2},
\end{aligned}$}
\end{equation*}
 $ a(t) $ and $ \Delta(t) $ are defined as in $ (\ref{at}) $ and $ (\ref{Delta}) $, respectively.
\end{lemma}
\begin{proof}
 Obviously,
\begin{equation*}
\begin{aligned}
&\frac{1}{2}\|(\alpha-1)(x(t)-x_t)+t^q \left(\dot{x}(t)+\gamma \nabla_{x}\mathcal{L}_t\left(x(t),y(t)+\theta(t)\dot{y}(t)\right)\right)\|^2\\
\leq& \frac{3}{2}(\alpha-1)^2\|x(t)-x_t\|^2+\frac{3}{2}t^{2q}\|\dot{x}(t)\|^2+\frac{3}{2}\gamma^2t^{2q}\|\nabla_{x}\mathcal{L}_t\left(x(t),y(t)+\theta(t)\dot{y}(t)\right)\|^2
\end{aligned}
\end{equation*}
and
\begin{equation*}
\begin{aligned}
&\frac{1}{2}\|(\alpha-1)(y(t)-y_t)+t^q \left(\dot{y}(t)-\gamma \nabla_{y}\mathcal{L}_t\left(x(t)+\theta(t)\dot{x}(t),y(t)\right)\right)\|^2\\
\leq& \frac{3}{2}(\alpha-1)^2\|y(t)-y_t\|^2+\frac{3}{2}t^{2q}\|\dot{y}(t)\|^2+\frac{3}{2}\gamma^2t^{2q}\|\nabla_{y}\mathcal{L}_t\left(x(t)+\theta(t)\dot{x}(t),y(t)\right)\|^2.
\end{aligned}
\end{equation*}
Therefore, for any $ M>0 $,
\begin{equation}\label{M}
\begin{aligned}
\frac{M}{t^q}\hat{\mathcal{E}}(t)\leq& M\left(t^{q+s}-2\gamma q t^{q-1}+\gamma\right)\left(\mathcal{L}_t(x(t),y_t)-\mathcal{L}_t(x_t,y(t))\right)\\
&+\frac{M}{2}(\alpha-1)\left((3\alpha-2)t^{-q}-qt^{-1}\right)\left(  \|x(t)-x_t\|^2+\|y(t)-y_t\|^2  \right)\\
&+\frac{3M}{2}t^{q}(\| \dot{x}(t)\|^2+\| \dot{y}(t)\|^2)+\frac{3M}{2}\gamma^2t^{q}\Delta(t).
\end{aligned}
\end{equation}

Now, we analyze the time derivative of $ \hat{\mathcal{E}}(t) $. 

Firstly, from (\ref{asd}), it is easy to see that
\begin{equation*}
\begin{aligned}
\frac{d}{dt}\mathcal{L}_t(x(t),y_t)=&\langle\nabla f(x(t)),\dot{x}(t) \rangle+\frac{c}{t^p}\langle x(t),\dot{x}(t) \rangle -\frac{cp}{2t^{p+1}}\|x(t)\|^2+\langle K^*y_t,\dot{x}(t) \rangle\\
&+\langle Kx(t),\dot{y}_t \rangle-\langle\nabla g(y_t),\dot{y}_t \rangle-\frac{c}{t^p}\langle y_t,\dot{y}_t \rangle +\frac{cp}{2t^{p+1}}\|y_t\|^2\\
=&\langle\nabla_x\mathcal{L}_t(x(t),y_t),\dot{x}(t)\rangle+\langle\nabla_y\mathcal{L}_t(x(t),y_t),\dot{y}_t\rangle-\frac{cp}{2t^{p+1}}\left(\|x(t)\|^2-\|y_t\|^2\right).
\end{aligned}
\end{equation*}
Similarly, we can show that
\begin{equation*}
\frac{d}{dt}\mathcal{L}_t(x_t,y(t))=\langle\nabla_x\mathcal{L}_t(x_t,y(t)),\dot{x}_t\rangle+\langle\nabla_y\mathcal{L}_t(x_t,y(t)),\dot{y}(t)\rangle-\frac{cp}{2t^{p+1}}\left(\|x_t\|^2-\|y(t)\|^2\right).
\end{equation*}
Consequently,
\begin{equation}\label{e1-1}\resizebox{0.925\hsize}{!}{$
\begin{aligned}
\dot{\hat{\mathcal{E}}}_1(t)=&\left((2q+s)t^{2q+s-1}-2\gamma q(2q-1)t^{2q-2}+\gamma q t^{q-1}\right)
\left(\mathcal{L}_t(x(t),y_t)-\mathcal{L}_t(x_t,y(t))\right)\\
&+\left(t^{2q+s}-2\gamma q t^{2q-1}+\gamma t^{q}\right)\Big( \big(\left \langle\nabla_x\mathcal{L}_t(x(t),y_t),\dot{x}(t)\right \rangle+\left \langle\nabla_y\mathcal{L}_t(x(t),y_t),\dot{y}_t\right \rangle\\
&-\frac{cp}{2t^{p+1}}(\|x(t)\|^2-\|y_t\|)^2 \big) -\big(\left \langle\nabla_x\mathcal{L}_t(x_t,y(t)),\dot{x}_t\right \rangle+\left \langle\nabla_y\mathcal{L}_t(x_t,y(t)),\dot{y}(t)\right \rangle\\ &-\frac{cp}{2t^{p+1}}(\|x_t\|^2-\|y(t)\|)^2\big)\Big).
\end{aligned}  $}
\end{equation}
Note that
\begin{equation}\label{nab-10}
\nabla_y\mathcal{L}_t(x(t),y_t)=\nabla_y\mathcal{L}_t(x_t,y_t)+K(x(t)-x_t)=K(x(t)-x_t),
\end{equation}
and
\begin{equation}\label{nab-100}
\nabla_x\mathcal{L}_t(x_t,y(t))= \nabla_x\mathcal{L}_t(x_t,y_t)+K^*(y(t)-y_t)=K^*(y(t)-y_t),
\end{equation}
where the last two equalities of (\ref{nab-10})  and (\ref{nab-100}) hold  due to (\ref{KKT}). Therefore, combining (\ref{e1-1}), (\ref{nab-10}) and (\ref{nab-100}),  and rearranging the terms, we get
\begin{equation}\label{e1-2}
\small{\begin{aligned}
\dot{\hat{\mathcal{E}}}_1(t)=&\left((2q+s)t^{2q+s-1}-2\gamma q(2q-1)t^{2q-2}+\gamma q t^{q-1}\right)
\left(\mathcal{L}_t(x(t),y_t)-\mathcal{L}_t(x_t,y(t))\right)\\
&+\left(t^{2q+s}-2\gamma q t^{2q-1}+\gamma t^{q}\right)\Big( \left \langle\nabla_x\mathcal{L}_t(x(t),y_t),\dot{x}(t)\right \rangle-\left \langle\nabla_y\mathcal{L}_t(x_t,y(t)),\dot{y}(t)\right \rangle \bigg.\\
&+\left \langle K(x(t)-x_t),\dot{y}_t\right \rangle -\left \langle K^*(y(t)-y_t),\dot{x}_t\right \rangle  \Big)-\frac{cp}{2}\left(t^{2q+s-p-1}-2\gamma q t^{2q-p-2}\right. \\
& \left.+\gamma t^{q-p-1}\right)\big((\|x(t)\|^2+\|y(t)\|^2)-(\|x_t\|^2+\|y_t\|^2)\big).
\end{aligned}}    
\end{equation}

Secondly, let $$
\bar{\mu}(t):=(\alpha-1)(x(t)-x_t)+t^q \left(\dot{x}(t)+\gamma \nabla_{x}\mathcal{L}_t\left(x(t),y(t)+\theta(t)\dot{y}(t)\right)\right).$$
Together with the first equality of   (\ref{dyn}), it follows that
\begin{equation*}
\dot{\bar{\mu}}(t)=(qt^{q-1}-1)\dot{x}(t)-(\alpha-1)\dot{x}_t+(\gamma qt^{q-1}-t^{q+s})\nabla_{x}\mathcal{L}_t\left(x(t),y(t)+\theta (t)\dot{y}(t)\right).
\end{equation*}
Moreover, by a similar argument as  given in Theorem \ref{Th1},  we derive that  for any $ t\geq t_0 $,
\begin{equation*}
\begin{aligned}
&\left \langle \bar{\mu}(t),\dot{\bar{\mu}}(t)\right \rangle\\
\leq&(\alpha-1)(qt^{q-1}-1)\left \langle x(t)-x_t,\dot{x}(t)\right \rangle+t^q(qt^{q-1}-1)\| \dot{x}(t)\|^2\\
&+(\alpha-1)(\gamma qt^{q-1}-t^{q+s})\Big(\mathcal{L}_t(x(t),y_t)-\mathcal{L}_t(x_t,y_t)+\frac{c}{2t^p}\|x(t)-x_t\|^2\Big)\\
&+(\alpha-1)(\gamma qt^{q-1}-t^{q+s})\left \langle K^*(y(t)-y_t+\theta(t)\dot{y}(t)),x(t)-x_t\right \rangle\\
&-(t^{2q+s}-2\gamma qt^{2q-1}+\gamma t^q)\left \langle\nabla_x\mathcal{L}_t(x(t),y_t),\dot{x}(t)\right \rangle\\
&-(t^{2q+s}-2\gamma qt^{2q-1}+\gamma t^q)\left \langle K^*(y(t)-y_t+\theta(t)\dot{y}(t)),\dot{x}(t)\right \rangle\\
&+\gamma t^q(\gamma qt^{q-1}-t^{q+s})\|\nabla_{x}\mathcal{L}_t(x(t),y(t)+\theta(t)\dot{y}(t))\|^2\\
&+\left \langle (\alpha-1)(x(t)-x_t)+t^q \left(\dot{x}(t)+\gamma \nabla_{x}\mathcal{L}_t\left(x(t),y(t)+\theta(t)\dot{y}(t)\right)\right)  , -(\alpha-1)\dot{x}_t \right \rangle
\end{aligned}
\end{equation*}
and
\begin{equation*}
\begin{aligned}
&\frac{d}{dt}\Big( \frac{\alpha-1}{2}(1-qt^{q-1})\|x(t)-x_t\|^2 \Big)\\
=&\frac{(\alpha-1)q(1-q)}{2}t^{q-2}\|x(t)-x_t\|^2\\
&+(\alpha-1)(1-qt^{q-1})(\left \langle x(t)-x_t,\dot{x}(t)\right \rangle-\left \langle x(t)-x_t,\dot{x}_t\right \rangle).
\end{aligned}
\end{equation*}
Therefore, for any $ t\geq t_0 $,
\begin{equation}\label{6-2}
\small{\begin{aligned}
\dot{\hat{\mathcal{E}}}_2(t)=&\left \langle \bar{\mu}(t),\dot{\bar{\mu}}(t)\right \rangle+\frac{d}{dt}\Big( \frac{\alpha-1}{2}(1-qt^{q-1})\|x(t)-x_t\|^2 \Big)\\
\leq &t^q(qt^{q-1}-1)\| \dot{x}(t)\|^2+(\alpha-1)(\gamma qt^{q-1}-t^{q+s})\Big(\mathcal{L}_t(x(t),y_t)-\mathcal{L}_t(x_t,y_t)\Big)\\
&+\Big(\frac{(\alpha-1)q(1-q)}{2}t^{q-2}+\frac{(\alpha-1)c}{2}(\gamma qt^{q-1-p}-t^{q+s-p})\Big)         \|x(t)-x_t\|^2   \\
&+(\alpha-1)(\gamma qt^{q-1}-t^{q+s})\left \langle K^*(y(t)-y_t+\theta(t)\dot{y}(t)),x(t)-x_t\right \rangle\\
&-(t^{2q+s}-2\gamma qt^{2q-1}+\gamma t^q)\left \langle\nabla_x\mathcal{L}_t(x(t),y_t),\dot{x}(t)\right \rangle\\
&-(t^{2q+s}-2\gamma qt^{2q-1}+\gamma t^{q})\left \langle K^*(y(t)-y_t+\theta(t)\dot{y}(t)),\dot{x}(t)\right \rangle\\
&+\gamma t^q(\gamma qt^{q-1}-t^{q+s})\|\nabla_{x}\mathcal{L}_t(x(t),y(t)+\theta(t)\dot{y}(t))\|^2\\
&-(\alpha-1)(\alpha-qt^{q-1})\left \langle x(t)-x_t,\dot{x}_t \right \rangle -(\alpha-1)t^q\left \langle \dot{x}(t),\dot{x}_t\right \rangle\\
&-(\alpha-1)\gamma t^q \left \langle \nabla_{x}\mathcal{L}_t(x(t),y(t)+\theta(t)\dot{y}(t)),\dot{x}_t \right \rangle.
\end{aligned}}
\end{equation}
Similarly, we can deduce that
\begin{equation}\label{6-3}
\small{\begin{aligned}
\dot{\hat{\mathcal{E}}}_3(t)\leq &t^q(qt^{q-1}-1)\| \dot{y}(t)\|^2+(\alpha-1)(\gamma qt^{q-1}-t^{q+s})\left(\mathcal{L}_t(x_t,y_t)-\mathcal{L}_t(x_t,y(t))\right)\\
&+\Big(\frac{(\alpha-1)q(1-q)}{2}t^{q-2}+\frac{(\alpha-1)c}{2}(\gamma qt^{q-1-p}-t^{q+s-p})\Big)         \|y(t)-y_t\|^2   \\
&-(\alpha-1)(\gamma qt^{q-1}-t^{q+s})\left \langle K(x(t)-x_t+\theta(t)\dot{x}(t)),y(t)-y_t\right \rangle\\
&+(t^{2q+s}-2\gamma qt^{2q-1}+\gamma t^q)\left \langle\nabla_y\mathcal{L}_t(x_t,y(t)),\dot{y}(t)\right \rangle\\
&+(t^{2q+s}-2\gamma qt^{2q-1}+\gamma t^q)\left \langle K(x(t)-x_t+\theta(t)\dot{x}(t)),\dot{y}(t)\right \rangle\\
&+\gamma t^q(\gamma qt^{q-1}-t^{q+s})\|\nabla_{y}\mathcal{L}_t((x(t)+\theta(t)\dot{x}(t)),y(t))\|^2\\
&-(\alpha-1)(\alpha-qt^{q-1})\left \langle y(t)-y_t,\dot{y}_t \right \rangle -(\alpha-1)t^q\left \langle \dot{y}(t),\dot{y}_t\right \rangle\\
&+(\alpha-1)\gamma t^q \left \langle \nabla_{y}\mathcal{L}_t((x(t)+\theta(t)\dot{x}(t)),y(t)),\dot{y}_t \right \rangle,
\end{aligned}}
\end{equation}
 for any $ t\geq t_0 $.
 
Finally, combining (\ref{theta}), (\ref{e1-2}), (\ref{6-2}) and(\ref{6-3}), we have
\begin{equation*}
\small{\begin{aligned}
\dot{\hat{\mathcal{E}}}(t)
\leq&a(t)\big(\mathcal{L}_t(x(t),y_t)-\mathcal{L}_t(x_t,y(t))\big)\\
&-\frac{cp}{2}\left(t^{2q+s-p-1}-2\gamma q t^{2q-p-2}+\gamma t^{q-p-1}\right)\big((\|x(t)\|^2+\|y(t)\|^2)-(\|x_t\|^2+\|y_t\|^2)\big)\\
&+  \frac{\alpha-1}{2}\left(q(1-q)t^{q-2}+c\gamma qt^{q-1-p}-ct^{q+s-p}      \right)  \big(\|x(t)-x_t\|^2 +\|y(t)-y_t\|^2\big)        \\
&+t^q(qt^{q-1}-1)(\| \dot{x}(t)\|^2+\| \dot{y}(t)\|^2)+\gamma t^q(\gamma qt^{q-1}-t^{q+s})\Delta(t)\\
&+\left(t^{2q+s}-2\gamma q t^{2q-1}+\gamma t^q \right)(\left \langle K(x(t)-x_t),\dot{y}_t\right \rangle -\left \langle K^*(y(t)-y_t),\dot{x}_t\right \rangle)\\
&-(\alpha-1)(\alpha-qt^{q-1})(\left \langle x(t)-x_t,\dot{x}_t \right \rangle+\left \langle y(t)-y_t,\dot{y}_t \right \rangle) \\
&-(\alpha-1)t^q(\left \langle \dot{x}(t),\dot{x}_t\right \rangle+\left \langle \dot{y}(t),\dot{y}_t\right \rangle) \\
&-(\alpha-1)\gamma t^q\left(\left \langle \nabla_{x}\mathcal{L}_t(x(t),y(t)+\theta(t)\dot{y}(t)),\dot{x}_t \right \rangle-\left \langle \nabla_{y}\mathcal{L}_t(x(t)+\theta(t)\dot{x}(t),y(t)),\dot{y}_t \right \rangle\right),
\end{aligned}}  
\end{equation*}
   for any $ t\geq t_0 $. Here, $ a(t) $ and $ \Delta(t) $ are defined as (\ref{at}) and (\ref{Delta}), respectively.
This, together with (\ref{M}), yields
\begin{equation}\label{dt-3}\resizebox{0.925\hsize}{!}{$
\begin{aligned}
&\frac{M}{t^q}\hat{\mathcal{E}}(t)+\dot{\hat{\mathcal{E}}}(t)\\
\leq&\big(M\left(t^{q+s}-2\gamma q t^{q-1}+\gamma\right)+a(t) \big)\left(\mathcal{L}_t(x(t),y_t)-\mathcal{L}(x_t,y(t))\right)\\
& -\frac{cp}{2}\left(t^{2q+s-p-1}-2\gamma q t^{2q-p-2}+\gamma t^{q-p-1}\right)\big((\|x(t)\|^2+\|y(t)\|^2)-(\|x_t\|^2+\|y_t\|^2)\big)\\
&+\Big(\frac{M}{2}(\alpha-1)\left((3\alpha-2)t^{-q}-qt^{-1}\right)+ \frac{\alpha-1}{2}(q(1-q)t^{q-2}\\
&+c\gamma qt^{q-1-p}-ct^{q+s-p}      ) \Big) (\|x(t)-x_t\|^2+\|y(t)-y_t\|^2)         \\
&+\Big(\Big(\frac{3M}{2}-1\Big)t^{q}+qt^{2q-1}\Big)(\| \dot{x}(t)\|^2+\| \dot{y}(t)\|^2)+\Big(\frac{3M}{2}\gamma^2t^{q}+\gamma^2q t^{2q-1}-\gamma t^{2q+s}\Big)\Delta(t)\\
&+\left(t^{2q+s}-2\gamma q t^{2q-1}+\gamma t^q\right)(\left \langle K(x(t)-x_t),\dot{y}_t\right \rangle -\left \langle K^*(y(t)-y_t),\dot{x}_t\right \rangle)\\
&-(\alpha-1)(\alpha-qt^{q-1})(\left \langle x(t)-x_t,\dot{x}_t \right \rangle+\left \langle y(t)-y_t,\dot{y}_t \right \rangle) -(\alpha-1)t^q(\left \langle \dot{x}(t),\dot{x}_t\right \rangle+\left \langle \dot{y}(t),\dot{y}_t\right \rangle)\\
&-(\alpha-1)\gamma t^q (\left \langle \nabla_{x}\mathcal{L}_t(x(t),y(t)+\theta(t)\dot{y}(t)),\dot{x}_t \right \rangle-\left \langle \nabla_{y}\mathcal{L}_t(x(t)+\theta(t)\dot{x}(t),y(t)),\dot{y}_t \right \rangle).
\end{aligned}    $}
\end{equation}

On the other hand, there exists $ t_1'\geq t_0 $ such that $$ t^{2q+s}-2\gamma q t^{2q-1}+\gamma t^q\geq 0, \forall t\geq t_1'.$$
Further, note that
\begin{equation*}
\small{\begin{split}
\left \langle K(x(t)-x_t),\dot{y}_t\right \rangle	\leq\frac{1}{2}\left( \frac{1}{\frac{4}{(\alpha-1)c}t^{q+p}\|K\|^2}\|K\|^2\|x(t)-x_t\|^2+\frac{4}{(\alpha-1)c}t^{q+p}\|K\|^2\|\dot{y}_t\|^2\right)
\end{split}}
\end{equation*}
		and
\begin{equation*}
\small{\begin{split}
		\left \langle K^*(y(t)-y_t),\dot{x}_t\right \rangle
			\leq\frac{1}{2}\left( \frac{1}{\frac{4}{(\alpha-1)c}t^{q+p}\|K^*\|^2}\|K^*\|^2\|y(t)-y_t\|^2+\frac{4}{(\alpha-1)c}t^{q+p}\|K^*\|^2\|\dot{x}_t\|^2\right).
\end{split}}
\end{equation*}
Thus, 
\begin{equation}\label{est-1}
\begin{split}
&\left(t^{2q+s}-2\gamma q t^{2q-1}+\gamma t^q\right)(\left \langle K(x(t)-x_t),\dot{y}_t\right \rangle -\left \langle K^*(y(t)-y_t),\dot{x}_t\right \rangle)\\
\leq&\frac{1}{2}\left(t^{2q+s}-2\gamma q t^{2q-1}+\gamma t^q\right)\Big(\frac{1}{\frac{4}{(\alpha-1)c}t^{q+p}} (\|x(t)-x_t\|^2+\|y(t)-y_t\|^2)\\
& +\frac{4}{(\alpha-1)c}t^{q+p}(\|K\|^2+\|K^*\|^2)(\|\dot{x}_t\|^2+\|\dot{y}_t\|^2) \Big)\\
=&\frac{(\alpha-1)c}{8}\left(t^{q+s-p}-2\gamma qt^{q-1-p}+\gamma t^{-p}\right)(\|x(t)-x_t\|^2+\|y(t)-y_t\|^2)\\
&+\frac{2}{(\alpha-1)c}\left(t^{3q+s+p}-2\gamma  qt^{3q-1+p}+\gamma t^{2q+p}\right)(\|K\|^2+\|K^*\|^2)(\|\dot{x}_t\|^2+\|\dot{y}_t\|^2),
\end{split}
\end{equation}
for all $ t\geq t_1' $. 

In a same manner as above, we can deduce that there exsits $ t_1\geq t_1' $ such that for all $ t\geq t_1 $,

\begin{equation}
\small{\begin{split}\label{est-2}
&-(\alpha-1)(\alpha-qt^{q-1})(\left \langle x(t)-x_t,\dot{x}_t \right \rangle+\left \langle y(t)-y_t,\dot{y}_t \right \rangle)\\
\leq&\frac{1}{2}(\alpha-1)(\alpha-qt^{q-1})\left(\frac{1}{4\frac{\alpha}{c}t^{p-s-q}}(\|x(t)-x_t\|^2+\|y(t)-y_t\|^2)\right.\\
&\left.~~~~+4\frac{\alpha}{c}t^{p-s-q}(\|\dot{x}_t\|^2+\|\dot{y}_t\|^2)\right)\\
=&\frac{(\alpha-1)c}{8}\left(t^{q+s-p}-\frac{q}{\alpha}t^{2q+s-1-p}\right)(\|x(t)-x_t\|^2+\|y(t)-y_t\|^2)\\
&+\frac{2(\alpha-1)\alpha}{c}\left(\alpha t^{p-s-q}-q t^{p-s-1}\right)(\|\dot{x}_t\|^2+\|\dot{y}_t\|^2),
\end{split}}
\end{equation}
\begin{equation}\label{est-3}
\small{\begin{split}
&-(\alpha-1)t^q(\left \langle \dot{x}(t),\dot{x}_t\right \rangle+\left \langle \dot{y}(t),\dot{y}_t\right \rangle)\\
\leq&\frac{1}{2}t^q(\| \dot{x}(t)\|^2+\| \dot{y}(t)\|^2)
+\frac{1}{2}(\alpha-1)^2t^{q}(\|\dot{x}_t\|^2+\|\dot{y}_t\|^2),  
\end{split}}
\end{equation}
and
\begin{equation}\label{est-4}
\small{\begin{aligned}
&-(\alpha-1)\gamma t^q (\left \langle \nabla_{x}\mathcal{L}_t(x(t),y(t)+\theta(t)\dot{y}(t)),\dot{x}_t \right \rangle-\left \langle \nabla_{y}\mathcal{L}_t(x(t)+\theta(t)\dot{x}(t),y(t)),\dot{y}_t \right \rangle)\\
\leq&\frac{1}{2}\gamma t^{2q+s}\Delta(t)+\frac{1}{2}\gamma(\alpha-1)^2  t^{-s}(\|\dot{x}_t\|^2+\|\dot{y}_t\|^2).
\end{aligned}}
\end{equation}
Combining (\ref{est-1}), (\ref{est-2}),  (\ref{est-3}) and  (\ref{est-4}) with (\ref{dt-3}), we have 
\begin{equation*}\resizebox{0.95\hsize}{!}{$
\begin{aligned}
&\frac{M}{t^q}\hat{\mathcal{E}}(t)+\dot{\hat{\mathcal{E}}}(t)\\
\leq&\Big(M\left(t^{q+s}-2\gamma q t^{q-1}+\gamma\right)+a(t) \Big)\left(\mathcal{L}_t(x(t),y_t)-\mathcal{L}(x_t,y(t))\right) -\frac{cp}{2}\left(t^{2q+s-p-1}\right.\\
& \left.-2\gamma q t^{2q-p-2}+\gamma t^{q-p-1}\right)\left((\|x(t)\|^2+\|y(t)\|^2)-(\|x_t\|^2+\|y_t\|^2)\right)\\
&+\Big(\frac{M(\alpha-1)(3\alpha-2)}{2}t^{-q}+\frac{(\alpha-1)c\gamma}{8}t^{-p}-\frac{M(\alpha-1)q}{2}t^{-1}-\frac{(\alpha-1)c}{4}t^{q+s-p}\\
&+\frac{(\alpha-1)c\gamma q }{4}t^{q-1-p}-\frac{(\alpha-1)qc}{8\alpha}t^{2q+s-1-p}+\frac{(\alpha-1)q(1-q)}{2}t^{q-2}\Big)(\|x(t)-x_t\|^2+\|y(t)-y_t\|^2) \\
&+\Big(\frac{3M-1}{2}t^{q}+qt^{2q-1}\Big)(\| \dot{x}(t)\|^2+\| \dot{y}(t)\|^2)+\Big(\frac{3M}{2}\gamma^2t^{q}+\gamma^2 q t^{2q-1}-\frac{1}{2}\gamma t^{2q+s}\Big)\Delta(t)\\
&+ \Big( \frac{2}{(\alpha-1)c}(\|K\|^2+\|K^*\|^2)\left(t^{3q+s+p}-2\gamma  qt^{3q-1+p}+\gamma t^{2q+p}\right)  \\
 &+\frac{2(\alpha-1)\alpha}{c}\left(\alpha t^{p-s-q}-q t^{p-s-1}\right)  +   \frac{1}{2}(\alpha-1)^2t^{q}+   \frac{1}{2}\gamma(\alpha-1)^2  t^{-s}
\Big)  (\|\dot{x}_t\|^2+\|\dot{y}_t\|^2),
\end{aligned}$}
\end{equation*}
  for any $ t\geq t_1 $. The proof is complete.
\end{proof}

Now, we establish the main result of this section, which demonstrates that both    the   convergence rate of the primal-dual gap and the strong convergence of the trajectory can be achieved simultaneously.
\begin{theorem}\label{slow} Let $ (x ,y ):\left[t_0, +\infty\right)\to\mathbb{R}^n \times \mathbb{R}^m  $ be a global solution of the dynamical system \textup{(\ref{dyn})} and let $ (x_t,y_t) $ be the saddle point of $ \mathcal{L}_t $. Suppose that $$0<p<1, ~p-2q-s<0  , ~   4q+s+p-2<0 , \mbox {and }~ 3q+s-p-1<0 .$$ Then, for $  (\bar{x}^*,\bar{y}^*)=\textup{Proj}_{\Omega}{0} $,
\begin{equation*}
\lim_{t\to +\infty} \|(x(t),y(t))-(\bar{x}^*,\bar{y}^*)\| =0,
\end{equation*}
and the following statements hold:
\begin{enumerate}
\item[{\rm (i)}] If $ q+2p\geq1 $, one has
\begin{equation*}
\mathcal{L}_t(x(t),y_t)-\mathcal{L}_t(x_t,y(t))= \mathcal{O}\left(\frac{1}{t^{2-2q-p}} \right),  ~~as~~   t\to +\infty ,
\end{equation*}and
\begin{equation*}
\|x(t)-x_t\|+\|y(t)-y_t\|= \mathcal{O}\left(\frac{1}{t^{1-2q-\frac{1}{2}(s+p)}} \right), ~~as~~   t\to +\infty.
\end{equation*}

\item[{\rm (ii)}] If $ q+2p<1 $, one has
\begin{equation*}
\mathcal{L}_t(x(t),y_t)-\mathcal{L}_t(x_t,y(t))= \mathcal{O}\left(\frac{1}{t^{1-q+p}} \right),  ~~as~~   t\to +\infty ,
\end{equation*}and
\begin{equation*}
\|x(t)-x_t\|+\|y(t)-y_t\|= \mathcal{O}\left(\frac{1}{t^{\frac{1}{2}(1+p-s)-\frac{3}{2}q}} \right),  ~~as~~   t\to +\infty .
\end{equation*}
\end{enumerate}

\end{theorem}
\begin{proof}
We start with the energy function $ \hat{\mathcal{E}}:\left[t_0,+\infty \right) \to \mathbb{R} $ defined  as (\ref{def}). It is easy to verify that there exists  $ t''_1\geq t_0 $ such that    $\hat{\mathcal{E}}(t)\geq0  $  for all $t\geq t''_1$.

Let $ \bar{z}^*=(\bar{x}^*,\bar{y}^*) $. By virtue of  Lemma \ref{lemma4.21} and (\ref{z1}), we have
 \begin{equation}\label{the41}
 \small{\begin{split}
 &\frac{M}{t^q}\hat{\mathcal{E}}(t)+\dot{\hat{\mathcal{E}}}(t)\\
 \leq&\Big(M\left(t^{q+s}-2\gamma q t^{q-1}+\gamma\right)+a(t) \Big)\left(\mathcal{L}_t(x(t),y_t)-\mathcal{L}(x_t,y(t))\right)\\
 & -\frac{cp}{2}\left(t^{2q+s-p-1}-2\gamma q t^{2q-p-2}+\gamma t^{q-p-1}\right)(\|x(t)\|^2+\|y(t)\|^2)\\
 &+l(t)\left(\|x(t)-x_t\|^2+\|y(t)-y_t\|^2\right) +\Big(\frac{3M-1}{2}t^{q}+qt^{2q-1}\Big)\left(\| \dot{x}(t)\|^2+\| \dot{y}(t)\|^2\right)\\
 &+\Big(\frac{3M}{2}\gamma^2t^{q}+\gamma^2 q t^{2q-1}-\frac{1}{2}\gamma t^{2q+s}\Big)\Delta(t)+ m(t)  \|\bar{z}^*\|^2,
 \end{split}}   
 \end{equation}
where
\begin{equation*}\resizebox{0.85\hsize}{!}{$
\begin{aligned}
 m(t)=& cp\left(t^{2q+s-p-1}-2\gamma q t^{2q-p-2}+\gamma t^{q-p-1}\right)\\ &+\frac{4p^2}{(\alpha-1)c}(\|K\|^2+\|K^*\|^2)\left(t^{3q+s+p-2}-2\gamma  qt^{3q+p-3}+\gamma t^{2q+p-2}\right)  \\ &+\frac{4(\alpha-1)\alpha p^2}{c}\left(\alpha t^{p-s-q-2}-q t^{p-s-3}\right)  +   (\alpha-1)^2p^2\left(t^{q-2}+   \gamma  t^{-s-2} \right).
\end{aligned} $}
\end{equation*}

We analyze the coefficients on the right hand side of the inequality (\ref{the41}). 

Clearly,
\begin{equation*}
\begin{aligned}
M\left(t^{q+s}-2\gamma q t^{q-1}+\gamma\right)+a(t)=&\left(M-(\alpha-1)\right)t^{q+s}+(\alpha-2M)\gamma q t^{q-1}\\
&+(2q+s)t^{2q+s-1}-2\gamma q(2q-1)t^{2q-2}
+ M\gamma.
\end{aligned}
\end{equation*}
 Now, set $ M<\min\{\frac{1}{3}, \alpha-1\} $. Together with $(\ref{coef0})$, $ \alpha>1 $ and $ M<\alpha-1 $, there exists $ t'_2\geq t''_1 $ such that
\begin{equation}\label{41the1}
 M\left(t^{q+s}-2\gamma q t^{q-1}+\gamma\right)+a(t)\leq 0, ~~~\forall  t\geq t'_2,
 \end{equation}
  and
\begin{equation}\label{41the2}
-\frac{cp}{2}\left(t^{2q+s-p-1}-2\gamma q t^{2q-p-2}+\gamma t^{q-p-1}\right)\leq0,  ~~~\forall  t\geq t'_2.
\end{equation}
Similarly, there exists $ t'_3\geq t'_2$ such that
\begin{equation}\label{41the3}
\frac{3M-1}{2}t^{q}+qt^{2q-1}\leq0 \mbox{ and }
\frac{3M}{2}\gamma^2t^{q}+\gamma^2 q t^{2q-1}-\frac{1}{2}\gamma t^{2q+s}\leq0, ~~\forall t\geq t'_3.
\end{equation}

Next, we consider the coefficient $ l(t) $. By $ 0<q<1 $, $ 0<p<1 $, $s>0 $ and  $ p-2q-s<0 $, we have $$ q+s-p>-p ,~
q+s-p>-q>-1>q-2,
$$ and
$$
q+s-p>2q+s-1-p>q-1-p.
$$
Since $ -\frac{(\alpha-1)c}{4}<0 $, it follows that there exists $ t'_4\geq t'_3 $ such that
\begin{equation}\label{41the4}
l(t)\leq 0, ~~\forall~ t\geq t'_4.
 \end{equation}
According to (\ref{Saddlepoint-t}),  we have $ \mathcal{L}_t(x(t),y_t)-\mathcal{L}_t(x_t,y(t))\geq0 $. Therefore, combining $(\ref{the41})$, $(\ref{41the1})$, $(\ref{41the2})$, $(\ref{41the3})$  and $(\ref{41the4})$,  and neglecting the nonpositive terms, we obtain that
\begin{equation}\label{45ty}
\frac{M}{t^q}\hat{\mathcal{E}}(t)+\dot{\hat{\mathcal{E}}}(t)\leq m(t)\|\bar{z}^*\|^2,  ~~~\forall t\geq t'_4.
\end{equation}

Now, we analyze the coefficient $ m(t) $ of $(\ref{45ty})$. It is easy to verify that
\begin{equation*}\left\{
\begin{array}{l}
2q+s-p-1>q-p-1>2q-p-2,\\
3q+s+p-2>q-2>-s-2,\\
3q+s+p-2>3q+p-3,\\
3q+s+p-2>2q+p-2>p-s-q-2>p-s-3.
\end{array}\right.
\end{equation*}
Thus,  we can deduce that there exists   $ t'_5\geq t'_4$ and a constant $C_0>0$ such that
\begin{equation*}
m(t)\|\bar{z}^*\|^2\leq C_0 (t^{3q+s+p-2}+t^{2q+s-p-1}),~~~\forall~t\geq t'_5.
\end{equation*}
Consequently, $(\ref{45ty})$ leads to
\begin{equation}\label{dt-1-1}
\frac{M}{t^q}\hat{\mathcal{E}}(t)+\dot{\hat{\mathcal{E}}}(t)\leq C_0 (t^{3q+s+p-2}+t^{2q+s-p-1}),~~\forall~t\geq t'_5.
\end{equation}
Multiplying $ e^{\frac{M}{1-q}t^{1-q}} $ on both sides of (\ref{dt-1-1}), we have
\begin{equation*}
\frac{d}{dt}\left(e^{\frac{M}{1-q}t^{1-q}}\hat{\mathcal{E}}(t) \right)\leq C_0 (t^{3q+s+p-2}+t^{2q+s-p-1})e^{\frac{M}{1-q}t^{1-q}},~~~\forall~t\geq t'_5.
\end{equation*}
Then, integrating it from $ t'_5 $ to $ t $, we obtain
\begin{equation}\label{dt-1-2}
e^{\frac{M}{1-q}t^{1-q}}\hat{\mathcal{E}}(t)\leq C_1+C_0\int_{t'_5}^{t} (\omega^{3q+s+p-2}+\omega^{2q+s-p-1})e^{\frac{M}{1-q}\omega^{1-q}}d\omega, ~~~\forall~t\geq t'_5,
\end{equation}
where $ C_1:= e^{\frac{M}{1-q}{t'_5}^{1-q}}\hat{\mathcal{E}}(t'_5) $.

In what follows,  we consider the right hand side of $(\ref{dt-1-2})$.
Using  the equation $(60)$ introduced in \cite{L2024S},  we deduce that there exist $ t'_6\geq t'_5 $, $ C_2>0$ and $ C_3>0 $ such that
\begin{equation}\label{dt-1-20}
t^{3q+s+p-2}e^{\frac{M}{1-q}t^{1-q}} \leq C_2 \frac{d}{dt}\left( t^{4q+s+p-2}e^{\frac{M}{1-q}t^{1-q}}\right), ~~~\forall  t\geq t'_6,
\end{equation}
and
\begin{equation}\label{dt-1-200}
t^{2q+s-p-1}e^{\frac{M}{1-q}t^{1-q}} \leq C_3 \frac{d}{dt}\left(t^{3q+s-p-1}e^{\frac{M}{1-q}t^{1-q}}\right),~~~\forall  t\geq t'_6.
\end{equation}
Obviously, from $(\ref{dt-1-20})$, we have
 \begin{align}\label{dt-1-2001}
 \int_{t'_5}^{t} \omega^{3q+s+p-2}e^{\frac{M}{1-q}\omega^{1-q}} d\omega  \leq &c'_1 +C_2 \int_{t'_6}^{t}\frac{d}{d\omega}\Big( \omega^{4q+s+p-2}e^{\frac{M}{1-q}\omega^{1-q}}\Big)d\omega&\nonumber\\
 =&c'_1 +C_2 \Big(t^{4q+s+p-2}e^{\frac{M}{1-q}t^{1-q}}-C_4\Big),&
 \end{align}
 where $c'_1:= C_2 \int_{t'_5}^{t'_6} \omega^{3q+s+p-2}e^{\frac{M}{1-q}\omega^{1-q}} d\omega  $ and $C_4:=t_{6}^{'4q+s+p-2}e^{\frac{M}{1-q}t_6'^{1-q}}.$
Similarly, we can deduce from $(\ref{dt-1-200})$ that
 \begin{equation}\label{dt-1-2002}
 \int_{t'_5}^{t} \omega^{2q+s-p-1}e^{\frac{M}{1-q}\omega^{1-q}} d\omega  \leq
 c'_2 +C_3 \left(t^{3q+s-p-1}e^{\frac{M}{1-q}t^{1-q}}-C_5\right),
 \end{equation}
 where $c'_2:=C_3 \int_{t'_5}^{t'_6} \omega^{2q+s-p-1}e^{\frac{M}{1-q}\omega^{1-q}} d\omega  $ and $C_5:=t_{6}^{'3q+s-p-1}e^{\frac{M}{1-q}t_6'^{1-q}}.$ By using $(\ref{dt-1-2001})$ and $(\ref{dt-1-2002})$,  $(\ref{dt-1-2})$ becomes
\begin{equation*}
\begin{aligned}
&e^{\frac{M}{1-q}t^{1-q}}\hat{\mathcal{E}}(t)\\
\leq &C_1+C_0\Big(c'_1+C_2 \left(t^{4q+s+p-2}e^{\frac{M}{1-q}t^{1-q}}-C_4\right)+c'_2+C_3\left(t^{3q+s-p-1}e^{\frac{M}{1-q}t^{1-q}}-C_5\right) \Big)\\
\leq&C+C_0\Big( C_2 t^{4q+s+p-2}+C_3 t^{3q+s-p-1}\Big)e^{\frac{M}{1-q}t^{1-q}},~~~\forall  t\geq t'_6,
\end{aligned}
\end{equation*}
where $C:=C_1+C_0c'_1+C_0c'_2$.
This means that
\begin{equation*}\label{defin}
\hat{\mathcal{E}}(t)
\leq\frac{C}{e^{\frac{M}{1-q}t^{1-q}} }+C_6\Big( t^{4q+s+p-2}+t^{3q+s-p-1}\Big),
~~~\forall~t\geq t'_6,
\end{equation*}
where $ C_6 := \max\{C_0C_2, C_0C_3\} $.

We now consider the following two cases:

\textbf{Case I}: $ q+2p\geq1 $. In this case,  $0>4q+s+p -2\geq 3q+s-p-1 $. Consequently,   there exists $ t'_7\geq t'_6 $ and $ C'_6>C_6 $ such that
\begin{equation*}
\hat{\mathcal{E}}(t)
\leq\frac{C}{e^{\frac{M}{1-q}t^{1-q}} }+C'_6 t^{4q+s+p-2},
~~~\forall~t\geq t'_7.
\end{equation*}
By $(\ref{def})$, it is easy to show that
\begin{equation*}
\mathcal{L}_t(x(t),y_t)-\mathcal{L}_t(x_t,y(t))= \mathcal{O}\left(\frac{1}{t^{2-2q-p}} \right),~~~\mbox{as}~t\to +\infty,
\end{equation*}
and
\begin{equation}\label{de23}
\|x(t)-x_t\|+\|y(t)-y_t\|= \mathcal{O}\left(\frac{1}{t^{1-2q-\frac{1}{2}(s+p)}} \right),~~~\mbox{as}~t\to +\infty.
\end{equation}

\textbf{Case II}: $ q+2p<1 $.  In this case,  $0>3q+s-p-1> 4q+s+p -2$. Therefore,  there exists $ t''_7\geq t'_6 $ and $ C''_6>C_6$ such that
\begin{equation*}
\hat{\mathcal{E}}(t)
\leq\frac{C}{e^{\frac{M}{1-q}t^{1-q}} }+C''_6 t^{3q+s-p-1},
~~~\forall~t\geq t''_7.
\end{equation*}
Taking into account $(\ref{def})$ again, we have
\begin{equation*}
\mathcal{L}_t(x(t),y_t)-\mathcal{L}_t(x_t,y(t))= \mathcal{O}\left(\frac{1}{t^{1-q+p}} \right),~~~\mbox{as}~t\to +\infty,
\end{equation*}
and
\begin{equation}\label{de230}
\|x(t)-x_t\|+\|y(t)-y_t\|= \mathcal{O}\left(\frac{1}{t^{\frac{1}{2}(1+p-s)-\frac{3}{2}q}} \right),~~~\mbox{as}~t\to +\infty.
\end{equation}

Finally, by  Lemma \ref{xtyt}, we have  $ \lim_{t\to +\infty} \|(x_t,y_t)-(\bar{x}^*,\bar{y}^*)\| =0$.  Then, it follows  from $(\ref{de23})$  and  $(\ref{de230})$  that
\begin{equation*}
\lim_{t\to +\infty} \|(x(t),y(t))-(\bar{x}^*,\bar{y}^*)\| =0.
\end{equation*}
 The proof is complete.
\end{proof}

By Theorem \ref{slow}, we have the following result.
\begin{corollary}\label{coro41}
 Let $ (x ,y ):\left[t_0, +\infty\right)\to\mathbb{R}^n \times \mathbb{R}^m  $ be a global solution of the dynamical system \textup{(\ref{dyn})}.
Assume that the basic hypotheses of Theorem $\ref{slow}$ are fulfilled. Then, for $ (\bar{x}^*,\bar{y}^*)=\textup{Proj}_{\mathrm{\Omega}}0$,  the following results hold.
\begin{enumerate}
\item[{\rm (i)}] If  $ q+2p\geq1 $, one has
 \begin{equation*}
 \mathcal{L}(x(t),\bar{y}^*)-\mathcal{L}(\bar{x}^*,y(t))=\mathcal{O}\left(\frac{1}{t^{1-2q-\frac{1}{2}(s+p)}}+\frac{1}{t^p} \right),~~~\textup{as}~t\to +\infty.
 \end{equation*}
\item[{\rm (ii)}]  If  $ q+2p<1 $, one has
 \begin{equation*}
  \mathcal{L}(x(t),\bar{y}^*)-\mathcal{L}(\bar{x}^*,y(t))=\mathcal{O}\left(\frac{1}{t^{\frac{1}{2}(1+p-s)-\frac{3}{2}q}}+\frac{1}{t^p} \right),~~~\textup{as}~t\to +\infty.
  \end{equation*}
\end{enumerate}
\end{corollary}
\begin{proof}
Let $ (x_t,y_t) $ be the saddle point of $ \mathcal{L}_t $.  Clearly,
\begin{align}\label{lt-1}
\mathcal{L}_t(x(t),y_t)-\mathcal{L}_t(x_t,y(t))=&\mathcal{L}_t(x(t),y_t)-\mathcal{L}_t(x_t,y_t)+\mathcal{L}_t(x_t,y_t)-\mathcal{L}_t(x_t,y(t))\nonumber\\
\geq&\mathcal{L}_t(x(t),y_t)-\mathcal{L}_t(\bar{x}^*,y_t)+\mathcal{L}_t(x_t,\bar{y}^*)-\mathcal{L}_t(x_t,y(t))\nonumber\\
\geq&\mathcal{L}(x(t),y_t)-\mathcal{L}(\bar{x}^*,y_t)+\mathcal{L}(x_t,\bar{y}^*)-\mathcal{L}(x_t,y(t))\\
&-\frac{c}{2t^p}(\|\bar{x}^*\|^2+\|\bar{y}^*\|^2),\nonumber
\end{align}
where the  two inequalities are deduced from  (\ref{Saddlepoint-t}) and (\ref{asd}), respectively.
Note that
\begin{align}\label{lt-3}
&\mathcal{L}(x(t),y_t)-\mathcal{L}(\bar{x}^*,y_t)+\mathcal{L}(x_t,\bar{y}^*)-\mathcal{L}(x_t,y(t))\nonumber\\
=&\mathcal{L}(x(t),\bar{y}^*)-\mathcal{L}(\bar{x}^*,y(t))+\left \langle K(x(t)-x_t),y_t-\bar{y}^*\right \rangle
+\left \langle K(\bar{x}^*-x_t),y(t)-y_t\right \rangle\nonumber\\
\geq&\mathcal{L}(x(t),\bar{y}^*)-\mathcal{L}(\bar{x}^*,y(t))-\| K\| \|x(t)-x_t\|\|y_t-\bar{y}^*\|\\
&-\| K\|\|\bar{x}^*-x_t\|\|y(t)-y_t\|.\nonumber
\end{align}
By combining (\ref{lt-1}) and (\ref{lt-3}), we have
\begin{equation}\label{lt-30}
\begin{aligned}
\mathcal{L}(x(t),\bar{y}^*)-\mathcal{L}(\bar{x}^*,y(t))
&\leq\mathcal{L}_t(x(t),y_t)-\mathcal{L}_t(x_t,y(t))+\frac{c}{2t^p}(\|\bar{x}^*\|^2+\|\bar{y}^*\|^2)\\
&~+\| K\|\left(\|x(t)-x_t\|\|y_t-\bar{y}^*\|+\|x_t-\bar{x}^*\|\|y(t)-y_t\|\right).
\end{aligned}
\end{equation}
Now,  we consider the following two cases:

\textbf{Case I}: $ q+2p\geq1 $.
In this case, by $ 4q+s+p-2<0 $, we have $$ 0<1-2q-\frac{1}{2}(p+s) <2-4q-s-p<2-2q-p.$$
 Combining (\ref{lt-30}),  Theorem \ref{slow} (i) and  $ \lim_{t\to +\infty} \|(x_t,y_t)-(\bar{x}^*,\bar{y}^*)\| =0$, and using the inequalities above, we can deduce that
 \begin{equation*}
 \mathcal{L}(x(t),\bar{y}^*)-\mathcal{L}(\bar{x}^*,y(t))=\mathcal{O}\left(\frac{1}{t^{1-2q-\frac{1}{2}(s+p)}}+\frac{1}{t^p} \right),~~~\textup{as}~t\to +\infty.
 \end{equation*}

\textbf{Case II}: $ q+2p<1 $.
In this case,  by $ 3q+s-p-1<0 $, we obtain 
\begin{equation*}
 0<\frac{1}{2}(1+p-s)-\frac{3}{2}q<1+p-s-3q<1+p-q .
 \end{equation*}
Combining  (\ref{lt-30}),  Theorem \ref{slow} (ii) and $ \lim_{t\to +\infty} \|(x_t,y_t)-(\bar{x}^*,\bar{y}^*)\| =0$, and using the inequalities above, we can deduce that
\begin{equation*}
      \mathcal{L}(x(t),\bar{y}^*)-\mathcal{L}(\bar{x}^*,y(t))=\mathcal{O}\left(\frac{1}{t^{\frac{1}{2}(1+p-s)-\frac{3}{2}q}}+\frac{1}{t^p} \right),~~~\textup{as}~t\to +\infty.
  \end{equation*}
The proof is complete.
\end{proof}
\begin{remark}
For solving the unconstrained optimization problem $(\ref{unconstrained})$, \cite[Theorem 10]{L2023O}, \cite[Theorem 2]{Attouch2023A} and \cite[Theorem 10]{L2024S} have demonstrated that, only under the setting of the parameters involved, the convergence rate of the objective function value and the strong convergence of the trajectory can be  achieved simultaneously.  Theorem \ref{slow} and Corollary \ref{coro41} show that, when solving  the convex-concave bilinear saddle point problem $(\ref{PD})$, the   convergence rate  of the primal-dual gap and the strong convergence of trajectory  can also be obtained simultaneously. Thus, Theorem \ref{slow} and Corollary \ref{coro41} can be regarded as an extension of this approach used in \cite{L2023O, Attouch2023A,L2024S} from solving unconstrained optimization problems to solving convex-concave bilinear saddle point problems.
\end{remark}

\section{Numerical experiments}
In this section,  we validate the obtained theoretical results by two numerical examples. In these experiments, the dynamical system $(\ref{dyn})$ is solved numerically with a Runge-Kutta adaptive method (ode45 in
MATLAB version R2019b).  All codes are performed  on a PC (with 2.30GHz Intel Core i5-8300H and 8GB memory).

\begin{example}\label{example5.1} Let $x:=(x_1,x_2)\in \mathbb{R}^2$ and $y:=(y_1,y_2)\in\mathbb{R}^2$.
\begin{equation}\label{example1}
\min_{x\in \mathbb{R}^2}\max_{y\in \mathbb{R}^2} (mx_1+nx_2)^2+(mx_1+nx_2)(jy_1+ky_2)-(jy_1+ky_2)^2,
\end{equation}
where $ m,$ $n$, $j$, $k\in \mathbb{R} $.
It is easy to show that the solution set of problem (\ref{example1}) is $$ \left\{(x,y)\in\mathbb{R}^2\times\mathbb{R}^2~|~mx_1+nx_2=0  \mbox{ and }  jy_1+ky_2=0\right\}.
$$
Clearly,  $(\bar{x}^*,\bar{y}^*)= (0,0,0,0) $ is the minimal norm   solution of problem (\ref{example1}). We consider the following two cases:
\begin{itemize}
\item $  m=1$,  $n=10$, $j=10$, $k=1 $.
\item $  m=2,$  $n=5$, $j=3$, $k=10 $.
\end{itemize}

In the first experiment, for any $(x^*,y^*)\in \mathrm{\Omega}$, we consider the  influence of Hessian-driven damping on the convergence of the primal-dual gap.

The dynamical system $(\ref{dyn})$ is solved on the time interval $[1, 50]$. Consider the initial time $ t_0=1 $ and take the following initial conditions:
\begin{equation}\label{initial}
x(t_0)=\left[\begin{matrix}
1\\
1.5
\end{matrix}\right],~~~y(t_0)=\left[\begin{matrix}
1\\
1.5
\end{matrix}\right],~~~\dot{x}(t_0)=\left[\begin{matrix}
1\\
1
\end{matrix}\right],~~~\mbox{and}~~~\dot{y}(t_0)=\left[\begin{matrix}
1\\
1
\end{matrix}\right].~~~
\end{equation}
 We set $ \alpha=2.5 $, $ q=0.42 $, $ s=0.005 $, $ p=0.268 $, $c=10$, which are satisfied with the assumptions in Theorem \ref{slow}.
 Figure \ref{f1} displays the behavior of $ \mathcal{L}(x(t),y^*)-\mathcal{L}(x^*,y(t)) $ along the trajectory  generated by the dynamical system (\ref{dyn}) under different settings on
 $ \gamma\in\{0, 0.8, 1.0, 1.2, 1.4\} $.
 \vspace{-1em}
 \begin{figure}[H]%
     \centering
     \subfloat[$  m=1,  n=10, j=10, k=1 $ ]{
         \includegraphics[width=0.48\linewidth]{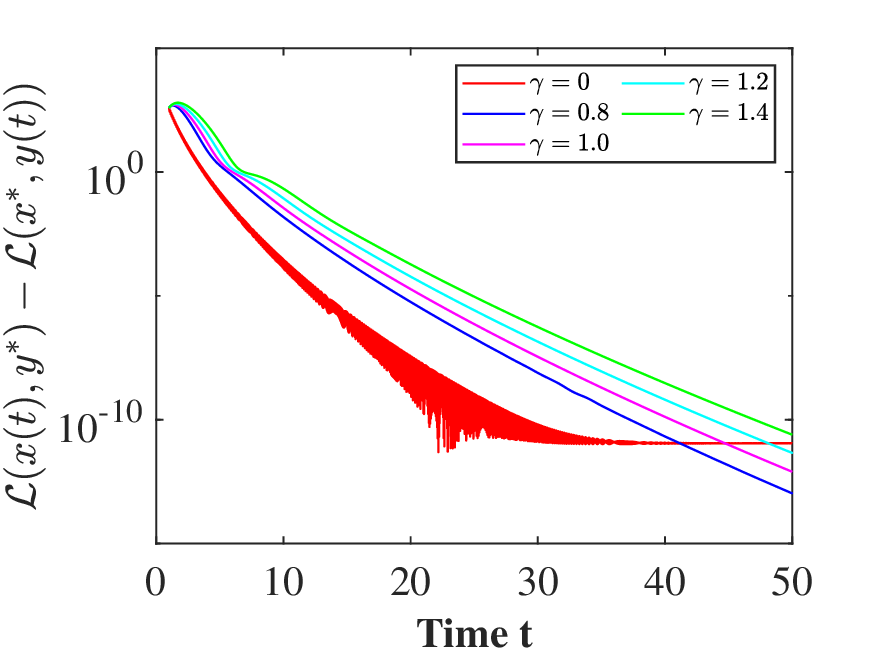}
         }\hfill
     \subfloat[$  m=2,  n=5, j=3, k=10 $]{
         \includegraphics[width=0.48\linewidth]{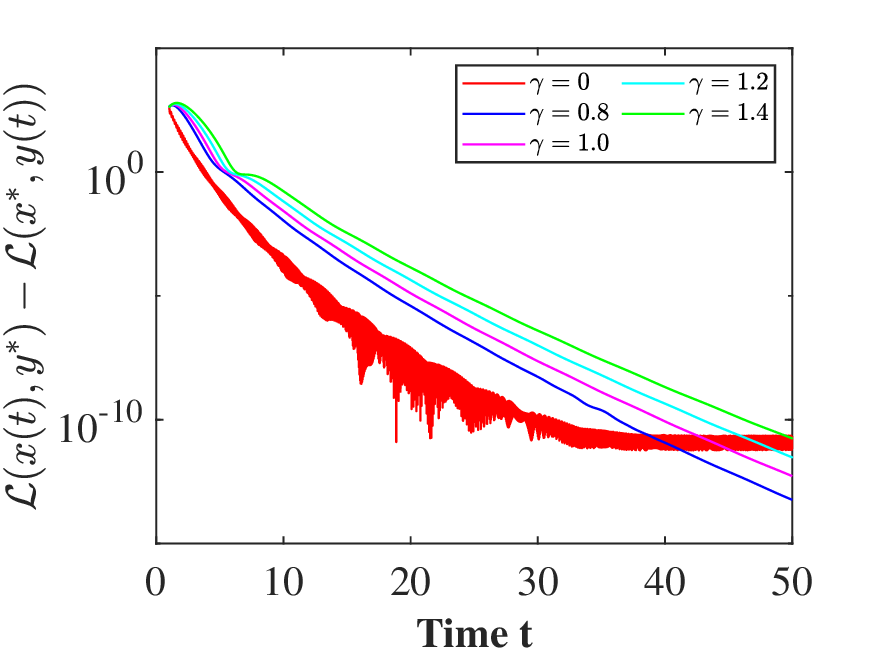}
         }
     \caption{Convergence  analysis of the dynamical system (\ref{dyn}) with different parameter $ \gamma $}
     \label{f1}
 \end{figure}\vspace{-1em}
   As shown in Figure \ref{f1}, the numerical results  are in  agreement with the theoretical claims. In each case,  $ \mathcal{L}(x(t),y^*)-\mathcal{L}(x^*,y(t)) $  converges more smoothly when the system is controlled by Hessian-driven damping.

In the second experiment,  we investigate the strong convergence of the trajectory to the minimal norm solution $(\bar{x}^*,\bar{y}^*)= (0,0,0,0)$  as well as the influence of Hessian-driven damping on the behaviors of trajectories.

Take $ \alpha=2.5 $, $ q=0.42 $, $ s=0.005 $ and $ p=0.268 $. Under the initial conditions (\ref{initial}), we plot the trajectory $ (x(t),y(t)) $ generated by the dynamical system (\ref{dyn}) in the following three cases:
\begin{itemize}
\item $ c=0$, $\gamma=0 $  (i.e., dynamical system (\ref{dyn}) without Tikhonov regularization and Hessian-driven damping).
\item $ c=0,$ $\gamma=0.8 $   (i.e.,  dynamical system (\ref{dyn}) without Tikhonov regularization).
\item $ c=10,$ $\gamma=0.8 $.
\end{itemize}\vspace{-1em}
\begin{figure}[H]
 \centering
    \subfloat[$ c=0,~\gamma=0 $ ]{
        \includegraphics[width=0.31\linewidth]{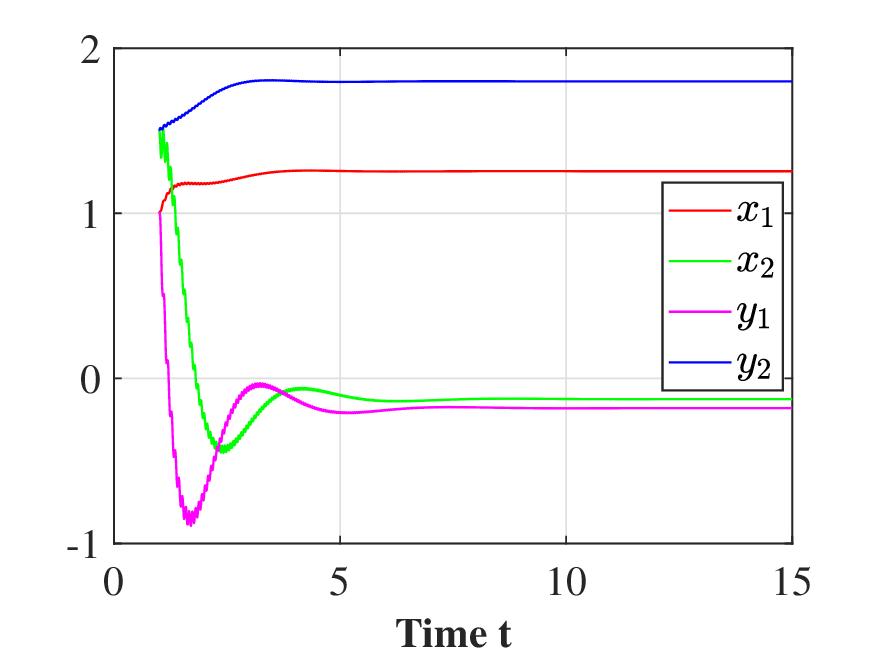}
        }\hfill
    \subfloat[$c=0,~ \gamma=0.8 $ ]{
        \includegraphics[width=0.31\linewidth]{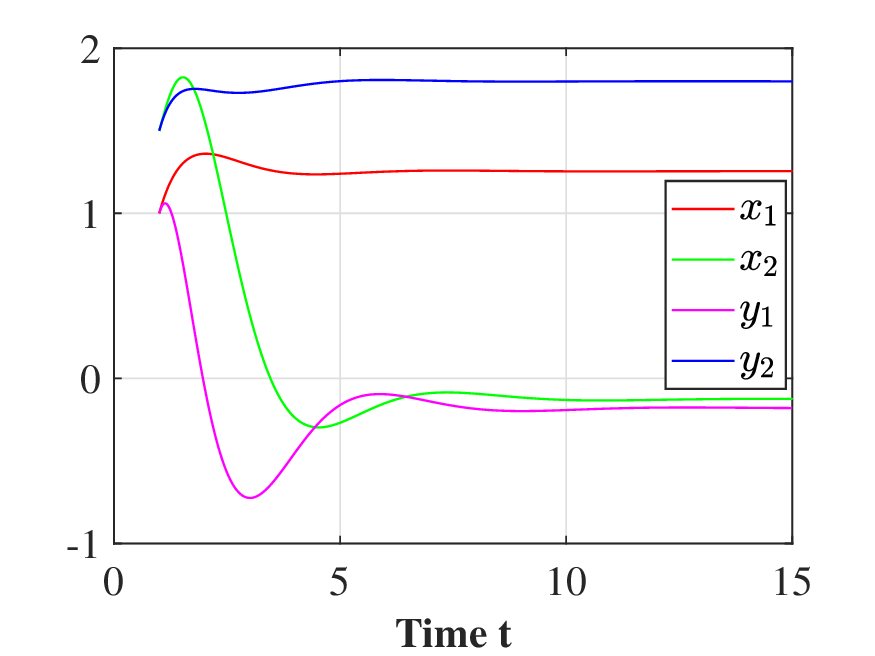}
        }
    \subfloat[$c=10,~ \gamma=0.8 $ ]{
            \includegraphics[width=0.31\linewidth]{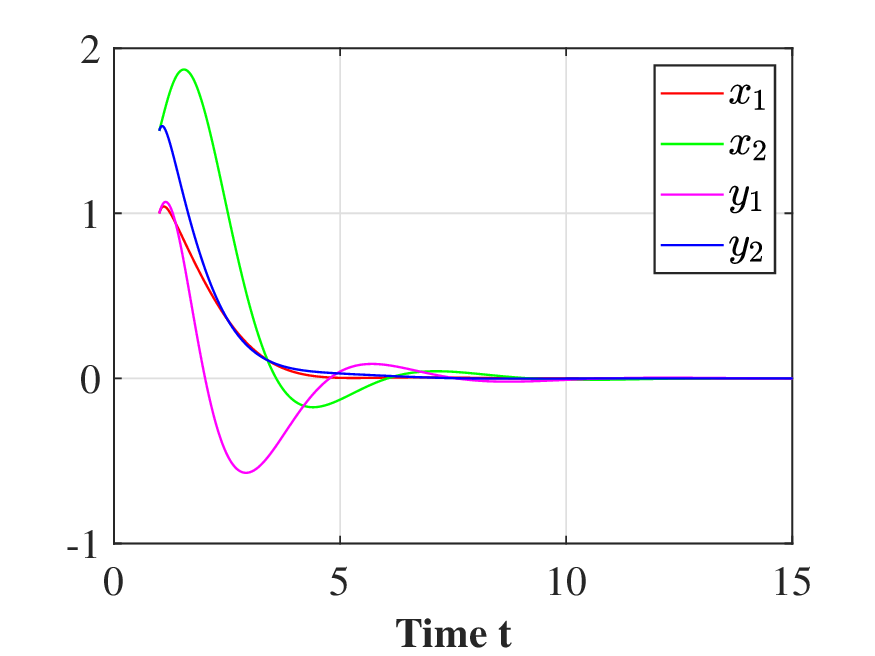}
            }
    \caption{Convergence of trajectories when $  m=1,  n=10, j=10, k=1 $}
    \label{f2}
\end{figure} \vspace{-2em}
\begin{figure}[H]
 \centering
    \subfloat[$ c=0,~\gamma=0 $ ]{
        \includegraphics[width=0.31\linewidth]{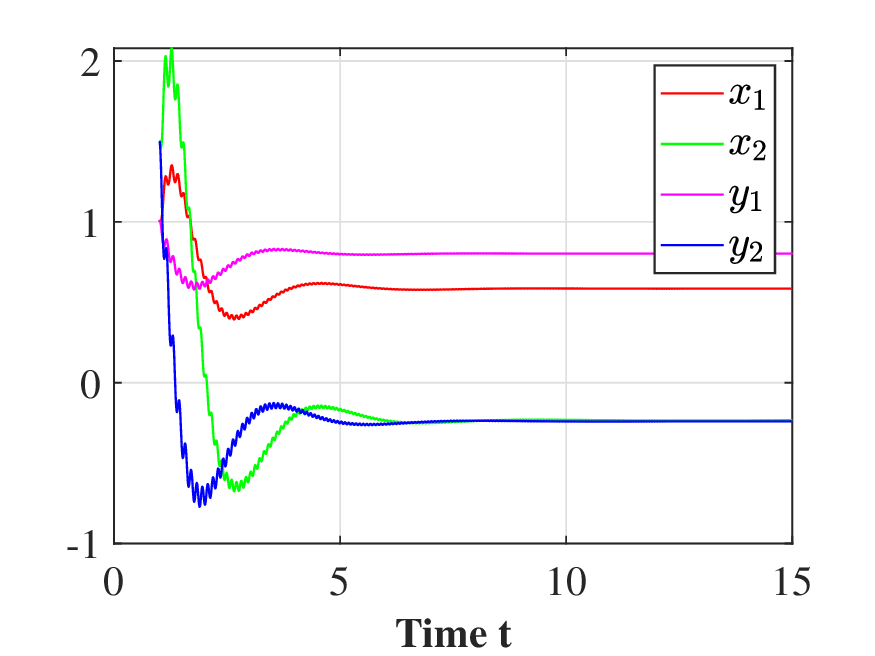}
        }\hfill
    \subfloat[$c=0,~ \gamma=0.8 $ ]{
        \includegraphics[width=0.31\linewidth]{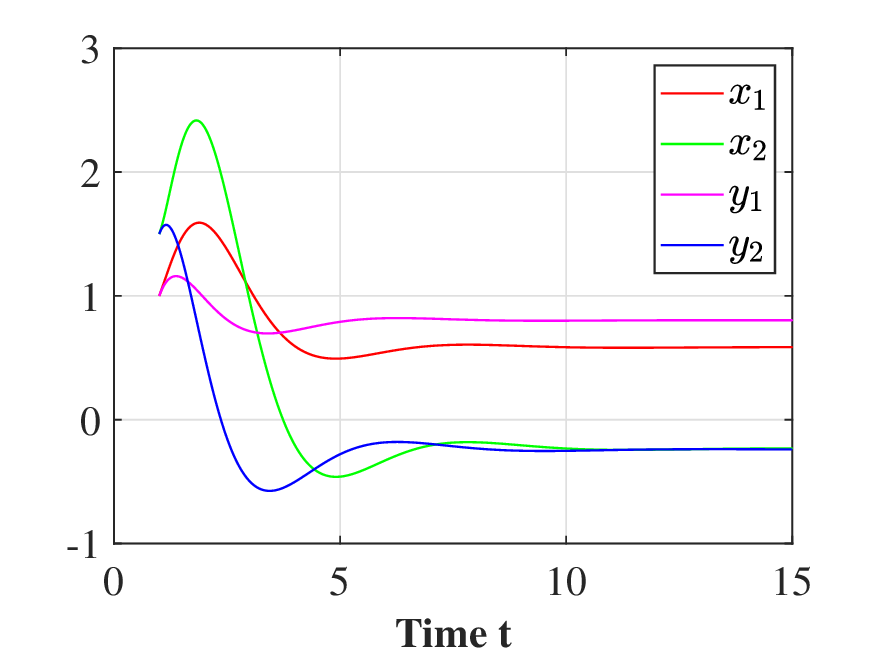}
        }
    \subfloat[$c=10,~ \gamma=0.8 $ ]{
            \includegraphics[width=0.31\linewidth]{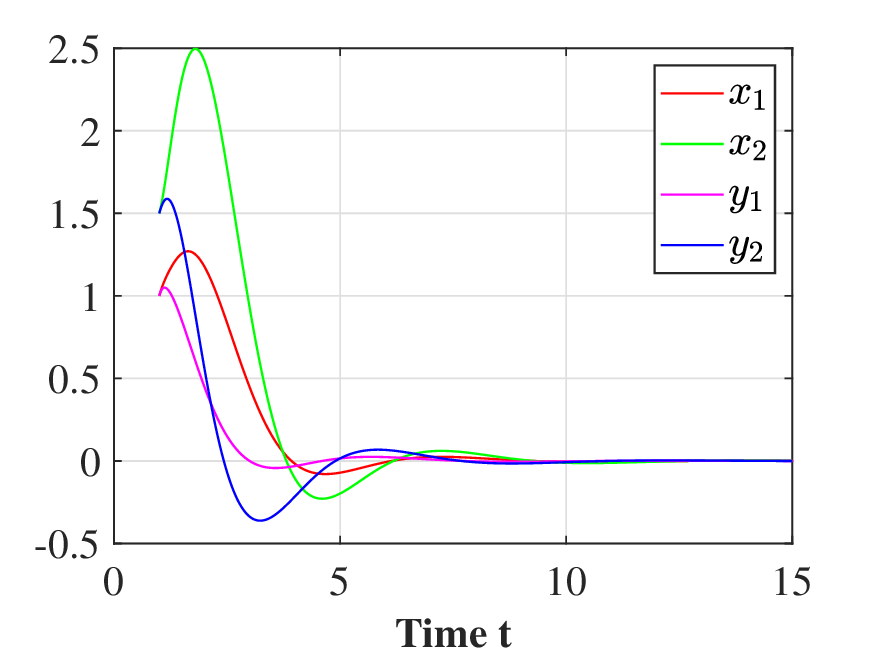}
            }
    \caption{Convergence of trajectories when $  m=2,  n=5, j=3, k=10 $}
    \label{f3}
\end{figure}

As shown in Figures \ref{f2} and \ref{f3}, we can observe that
 \begin{itemize}
\item[{\rm (i)}] Only when the dynamical system (\ref{dyn}) is controlled by Tikhonov regularization  does the trajectory   $ (x(t),y(t)) $ converge  to the minimal norm solution $ (\bar{x}^*,\bar{y}^*)=(0,0,0,0) $.

\item[{\rm (ii)}] When the dynamical system (\ref{dyn}) is controlled by Hessian-driven damping (i.e., with $ \gamma=0.8 $),  the trajectory $ (x(t),y(t)) $ converges more smoothly compared to the case  where $ \gamma=0 $.

    \end{itemize}

\end{example}

\begin{example}\label{example5.2} Let $x\in \mathbb{R}^n$ and $y\in\mathbb{R}^m$. Consider the following $ \ell_2 $-regularized problem:
\begin{equation} \label{exp5.2}
\min_{x\in \mathbb{R}^n} \Phi(x)=\frac{1}{2}\|Kx-b\|^2+\eta \|x\|^2,
\end{equation}
where $ K\in \mathbb{R}^{m\times n} $   and
$ b\in\mathbb{R}^m $. Equivalently, its saddle point formulation reads as:
\begin{equation*}\label{example5.2-1}
\min_{x\in \mathbb{R}^n}\max_{y\in \mathbb{R}^m} \eta \|x\|^2+\left \langle Kx,y\right \rangle-\Big(\frac{1}{2}\|y\|^2+\left \langle b,y\right \rangle\Big).
\end{equation*}

In the following numerical experiment, we are interested in  the  influence of Hessian-driven damping on the convergence of objective function value of problem (\ref{exp5.2}).

The dynamical system $(\ref{dyn})$ is solved on the time interval $[1, 200]$. All entries of initial conditions, $ K $ and $ b $ are generated from the standard Gaussian distribution. Moreover, three different dimensional settings are considered:
\begin{itemize}
\item  $ m=20 $, $ n=50 $.
\item  $ m=50 $, $ n=100 $.
\item  $ m=100 $, $ n=200 $.
\end{itemize}

  Take $ \eta=1 $, $ \alpha=3 $, $ c=5 $, $s=0.4$ and $ p=2.3 $. We test the dynamical system $(\ref{dyn})$   under the following settings on parameters $q$ and $\gamma$:
\begin{itemize}
\item   $ q=0.6,$ $\gamma=0 $.
\item    $ q=0.7,$ $\gamma=0 $.
\item   $ q=0.8,$ $\gamma=0 $.
\item     $ q=0.6,$ $\gamma=0.2 $.
\item    $ q=0.7,$ $\gamma=0.2 $.
\item    $ q=0.8,$ $\gamma=0.2 $.
\end{itemize}

 Figure \ref{HD-exp2} depicts the behaviors of the objective error $ \Phi(x(t))-\Phi(x^*)$  along the trajectory  generated by the dynamical system (\ref{dyn}) with different dimensions.
 \vspace{-1em}
 \begin{figure}[H]%
     \centering
         \subfloat[$m=20,~ n=50 $ ]{
             \includegraphics[width=0.31\linewidth]{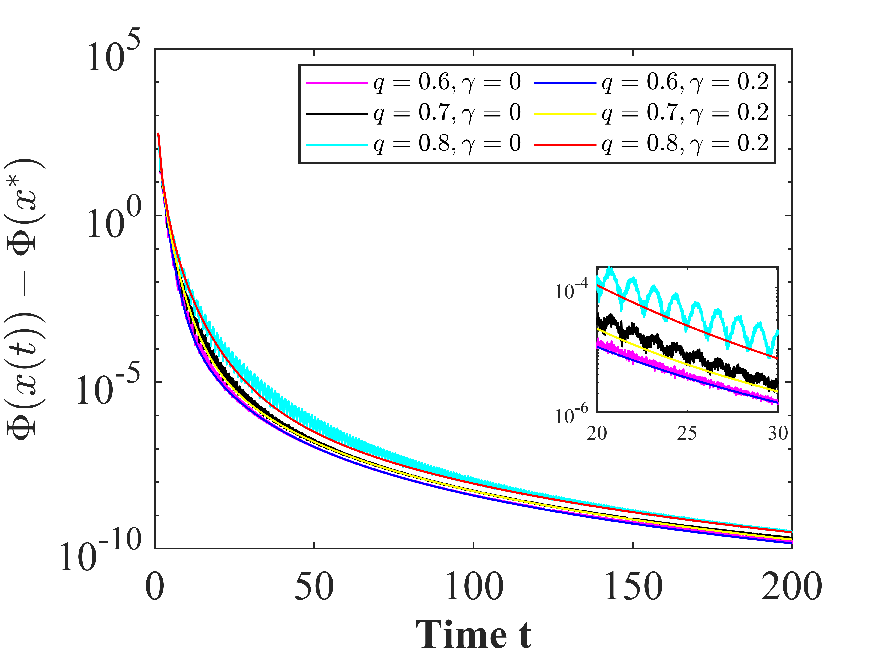}
             }
         \subfloat[$m=50,~ n=100 $ ]{
                 \includegraphics[width=0.31\linewidth]{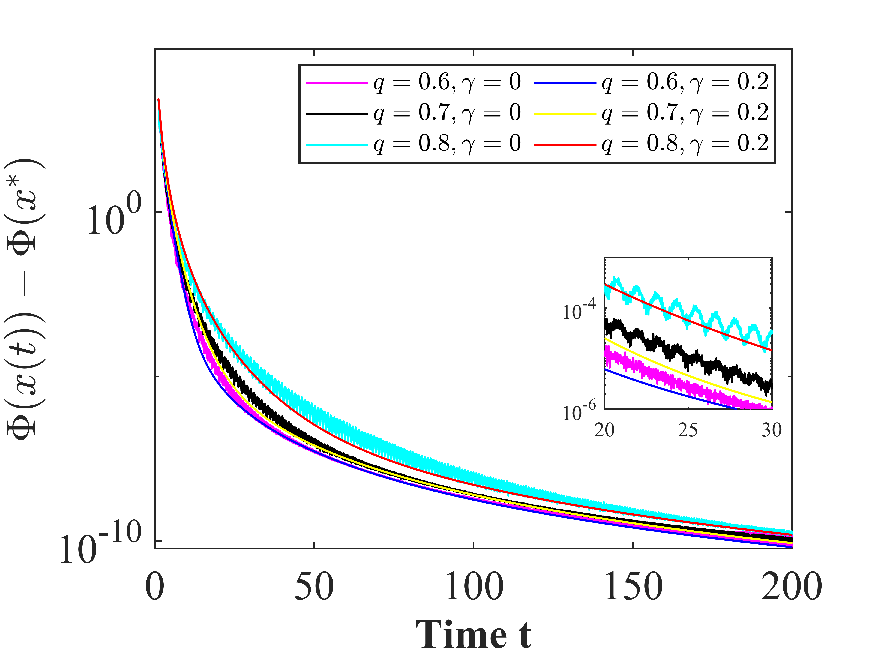}
                 }
         \subfloat[$m=100,~ n=200 $ ]{
                  \includegraphics[width=0.31\linewidth]{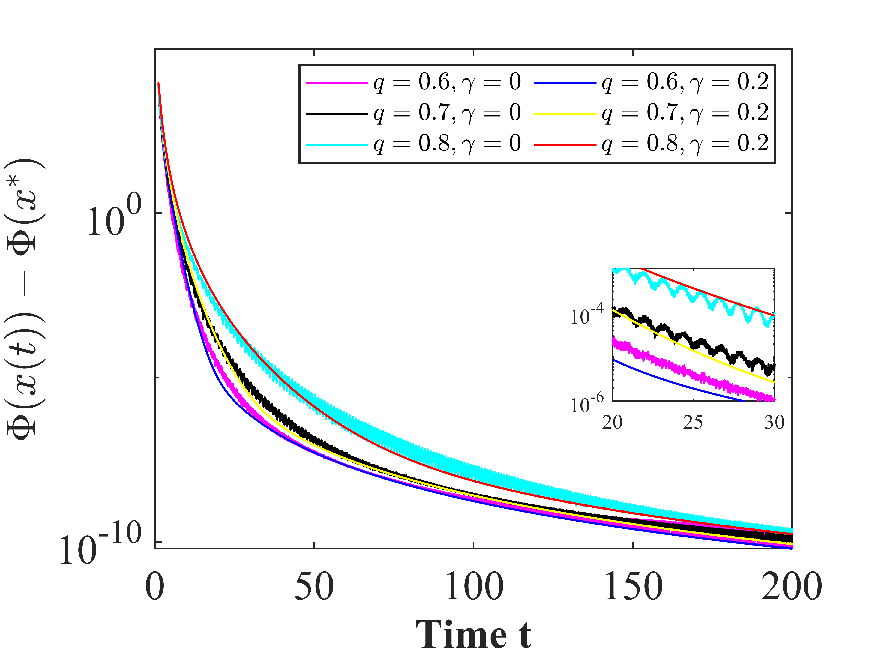}
                 }
     \caption{Convergence  analysis of the dynamical system (\ref{dyn}) under different dimensional settings}
     \label{HD-exp2}
 \end{figure}\vspace{-1em}

As shown in Figure \ref{HD-exp2}, under different dimensional settings, the dynamical system (\ref{dyn}) controlled by Hessian-driven damping can significantly reduce oscillations.

\end{example}

\section{Conclusion}
 In this paper, we introduce a Tikhonov regularized second-order primal-dual dynamical system   (\ref{dyn})  to solve the convex-concave bilinear saddle point problem $(\ref{PD})$. Compared with the dynamical systems introduced in \cite{2024he,ding,ours} for solving the convex-concave bilinear saddle point problem $(\ref{PD})$, the  system (\ref{dyn})  not only encompasses slow  viscous damping, extrapolation and time scaling, but  is also controlled by a  Hessian-driven damping. By using appropriate damping, extrapolation,  scaling and Tikhonov regularization  coefficients, we establish the fast convergence rates of the values  and the strong convergence property of the generated trajectory. Through numerical experiments, we observe that the dynamical system involved in Hessian-driven damping can neutralize the oscillations occurring during iterations.

In the future, a  natural  direction for research is to investigate the explicit discretization  of the  system (\ref{dyn}), which leads to an inertial algorithm of gradient type that incorporates a Tikhonov regularization term for solving the convex-concave bilinear saddle point problem $(\ref{PD})$. Furthermore, based on the results obtained in the present paper, we can  also explore the strong convergence properties of the corresponding  algorithm. On the other hand, as we all know, multi-objective optimization is also  an active area of research in recent years and has received extensive attention by many researchers \cite{chu18,sunjota21,jcm24,sc24,jota25g}. It is also of importance to consider how the second-order dynamical system approach for solving multi-objective optimization \cite{att15jmaa,jota24s,siam24} can be extended to handle vector convex-concave saddle point problems.



\section*{Funding}
\small{  This research is supported by the Natural Science Foundation of Chongqing (CSTB2024NSCQ-MSX0651 and CSTB2024NSCQ-MSX1282) and the Team Building Project for Graduate Tutors in Chongqing (yds223010).}

\section*{Data availability}

 \small{ The authors confirm that all data generated or analysed during this study are included in this article.}

 \section*{Declaration}

 \small{\textbf{Conflict of interest} No potential conflict of interest was reported by the authors.}

\end{document}